\documentclass[11pt]{article}
\usepackage{fullpage,amsthm,graphics,hyperref,bm,xcolor,verbatim,amssymb, amsmath} 
\usepackage{caption}
\usepackage{tikz}
\usetikzlibrary{snakes}

\usepackage{hyperref}
\hypersetup{colorlinks=true,linkcolor=black,anchorcolor=blue,citecolor=black,filecolor=blue,urlcolor=blue,bookmarksnumbered=true,pdfview=FitB}

\RequirePackage{marginnote,hyperref}
\newcommand{\aside}[1]{\marginnote{\scriptsize{#1}}[0cm]}
\newcommand{\aaside}[2]{\marginnote{\scriptsize{#1}}[#2]}
\newcommand\Emph[1]{\emph{#1}\aside{#1}}
\newcommand\EmphE[2]{\emph{#1}\aaside{#1}{#2}}
\newcommand\DeltaG{\Delta}
\newcommand\LL{\mathcal{L}}
\newcommand\CC{\mathcal{C}}

\newtheorem{lem}{Lemma}
\newtheorem{lemi}{Lemma}
\newtheorem{cor}[lem]{Corollary}

\newtheorem{thm}[lem]{Theorem}
\theoremstyle{definition}

\newtheorem{defn}[lem]{Definition}
\newtheorem{clm}{Claim}

\newtheorem{mainthm}{Main Theorem}
\newtheorem{struct-lem}{Structural Lemma}
\newtheorem{struct-lemi}{Structural Lemma}
\title{Kempe Equivalent List Edge-Colorings of Planar Graphs}
\def\vph{\varphi}
\newcommand\Z{\mathbb{Z}}
\def\dbox{\!\!\!\!\!\qed\,}

\newenvironment{clmproof}[1]{\par\noindent\underline{Proof.}\space#1}{\leavevmode\unskip\penalty9999\hbox{}\nobreak\hfill\quad\hbox{$\diamondsuit$}\smallskip}

\begin{document}
\author{Daniel W. Cranston\thanks{Department of Computer Science, Virginia
Commonwealth University, Richmond, VA, USA; \texttt{dcranston@vcu.edu}}}
\maketitle
~\vspace{-.5in}

\begin{center}
\textit{\small{Dedicated to the memory of Landon Rabern.}}
\end{center}

\begin{abstract}
For a list assignment $L$ and an $L$-coloring $\vph$, a Kempe swap in $\vph$ is
\emph{$L$-valid} if it yields another $L$-coloring.  Two $L$-colorings are
\emph{$L$-equivalent} if we can form one from another by a sequence of $L$-valid Kempe
swaps.  And a graph $G$ is \emph{$L$-swappable} if every two of its
$L$-colorings are $L$-equivalent.  We consider $L$-swappability of line graphs
of planar graphs with large maximum degree.  
Let $G$ be a planar graph with $\Delta(G)\ge 9$ and let $H$ be the line graph
of $G$. If $L$ is a $(\Delta(G)+1)$-assignment to $H$, then $H$ is $L$-swappable. 
Let $G$ be a planar graph with $\Delta(G)\ge 15$ and let $H$ be the line graph
of $G$. If $L$ is a $\Delta(G)$-assignment to $H$, then $H$ is $L$-swappable. 
The first result is analogous to one for $L$-choosability by Borodin, which was
later strengthened by Bonamy.
The second result is analogous to another for $L$-choosability by Borodin, which was
later strengthened by Borodin, Kostochka, and Woodall.
\end{abstract}

\section{Introduction and Main Results}

For a graph $G$ with a proper $k$-coloring $\vph$ and distinct
$i,j\in\{1,\ldots,k\}$,
an \emph{$(i,j)$-Kempe swap} interchanges the colors on some component of the
subgraph induced by colors $i$ and $j$.  Performing a Kempe swap on $\vph$
yields another proper $k$-coloring of $G$.  Kempe swaps were introduced in the
late 1800s in an attempt to prove the 4 Color Theorem.  This inital effort
failed, but Heawood largely salvaged the idea, and used it to prove the 5 Color
Theorem.  Kempe swaps did play an important role in the eventual proof of
the 4 Color Theorem nearly a century later, and they have also been a central tool
in most work on edge-coloring.  Two $k$-colorings are (Kempe)
\emph{$k$-equivalent} if we can transform one into the other by a sequence of Kempe
swaps, never using more than $k$ colors.  Mohar conjectured that if $G$ is
$k$-regular, then all of its $k$-colorings are $k$-equivalent.  This was proved
for $k=3$ by Feghali, Johnson, and Paulusma~\cite{FJP} (with a single
exception, the Cartesian product $K_3\dbox K_2$), and for $k\ge 4$ 
by Bonamy, Bousquet, Feghali, and Johnson~\cite{BBFJ}.

Here we study edge-coloring; equivalently, this is coloring the line graph.  
Vizing proved that $\chi'\le \Delta+1$
for every graph (where $\chi'$ and $\Delta$ denote edge-chromatic number and
maximum degree, respectively).  Mohar~\cite{mohar} extracted from Vizing's proof an algorithm
showing that all $k$-edge-colorings of a graph are $k$-equivalent when
$k\ge\chi'+2$.  He also asked whether the same result holds when $k=\chi'+1$.  This problem
was investigated by McDonald, Mohar, and Scheide~\cite{MMS} for graphs with
$\Delta\le 4$ and by Asratian and Casselgren~\cite{AC} for graphs with
$\Delta\le 5$.
One consequence of Main Theorem~\ref{main-thm1} is that we answer Mohar's question 
affirmatively for all planar graphs with $\Delta\ge 9$.
One consequence of Main Theorem~\ref{main-thm2} is that, for all planar graphs with
$\Delta\ge 15$, we get the stronger statement that all $k$-edge-colorings are
$k$-equivalent whenever $k\ge \chi'$.

A \emph{list assignment} $L$ for a graph $G$ gives each vertex $v\in V(G)$ a set of
(allowable) colors $L(v)$, and an \emph{$L$-coloring} is a proper coloring
$\vph$ such that $\vph(v)\in L(v)$ for all vertices $v$.  If $|L(v)|=k$ for all
$v\in V(G)$, and some constant $k$, then $L$ is a \emph{$k$-assignment}.
Historically, Kempe swaps have not
been studied for list coloring, since we encounter the obvious concern that a
Kempe swap in a proper $L$-coloring of $G$ might not yield another proper
$L$-coloring of $G$.  (Consider a single edge $vw$ with $L(v)=\{1,2\}$,
$L(w)=\{1,3\}$, $\vph(v)=1$, $\vph(w)=3$, and a 1,3-swap at $w$.)  However,
we can easily sidestep this obstacle.  A Kempe swap is \emph{valid} for a graph
$G$, list assignment $L$, and $L$-coloring $\vph$ if performing that Kempe swap
in $\vph$ yields another proper $L$-coloring of $G$.  For short, we say the swap
is \Emph{$L$-valid} for $\vph$.
Two $L$-colorings are now \Emph{$L$-equivalent} if one can be transformed to the
other by a sequence of $L$-valid Kempe swaps, and $G$ is
\EmphE{$L$-swappable}{4mm} if every two of its $L$-colorings are
$L$-equivalent.\footnote{We first read about
$L$-swappability near the end of~\cite{feghali}.  For a planar graph $G$ and a
5-assignment $L$ for $G$, Feghali asked whether $G$ is necessarily
$L$-swappable.  As far as we know, this question remains open.} Cranston and Mahmoud~\cite{CM} showed that if $G$ is $k$-regular, with $k\ge
3$, and $L$ is a $k$-assignment for $G$, then $G$ is $L$-swappable (again with
the single exception of $K_3\dbox K_2$).

If $L$ is a list-assignment with $|L(v)|=d(v)$ for all vertices $v$, then $L$ is a
\emph{degree assignment}.  
If, for each degree assignment $L$, graph $G$ has at
least one $L$-coloring, then $G$ is \emph{degree-choosable}.
If $G$ is $L$-swappable whenever $L$ is a degree assignment, then $G$ is
\emph{degree-swappable}.  
A key step in the proof of Cranston and Mahmoud is showing that if $H$ is
degree-swappable and $H$ is an induced subgraph of $G$, then $G$ is also
degree-swappable.  This is perhaps unsurprising, since an analogous result holds
for degree-choosability.  However, we should note that being degree-swappable is
more restrictive than being degree-choosable.  For example, all even cycles are
2-choosable.  However, if $G$ is a 4-cycle $vwxy$ with $L(v)=\{1,2\}$,
$L(w)=\{2,3\}$, $L(x)=\{3,4\}$, and $L(y)=\{4,1\}$, then $G$ has two
$L$-colorings, but neither of these admits any $L$-valid Kempe swap.  Thus, $G$
is not degree-swappable.  In fact, no cycle is degree-swappable.  Further,
if we begin with any Gallai tree (a connected graph where each block is an odd
cycle or a clique) and add any edge, the resulting graph is degree-choosable,
but is not degree-swappable.

In this paper we investigate the swappability of the line graph $H$
of a graph $G$, which we typically call the \emph{edge-swappability} of $G$.  
We are not aware of previous work on this problem.\footnote{Ito, Kami\'{n}ski, and
Demaine~\cite{IKD} studied a similar problem, but they only allowed recoloring a
single edge at each step, which is a far more restrictive model.}
We focus on the case that $G$ is a simple planar graph.
Borodin~\cite{borodin} showed
that if such a $G$ has $\Delta(G)\ge 9$, then
every $(\Delta(G)+1)$-assignment $L$ to the line graph of $G$ admits an
$L$-coloring.  Bonamy~\cite{bonamy} later extended this result to such graphs
with $\Delta(G)\ge 8$. Borodin~\cite{borodin} also showed that if
$G$ is such a graph with $\Delta(G)\ge 14$, then every
$\Delta(G)$-assignment $L$ to the line graph of $G$ admits an $L$-coloring.
Borodin, Kostochka, and Woodall~\cite{BKW} later extended this result to 
such graphs with $\Delta(G)\ge 12$.
Here we prove two analogous results for swappability.

\begin{mainthm}
\label{main-thm1}
Let $G$ be a simple planar graph with $\Delta(G)\ge 9$.  If $H$ is the line graph of
$G$ and $L$ is a $(\Delta(G)+1)$-assignment for $H$, then all $L$-colorings of
$H$ are $L$-equivalent.
\end{mainthm}

\begin{mainthm}
\label{main-thm2}
Let $G$ be a simple planar graph with $\Delta(G)\ge 15$.  If $H$ is the line graph of
$G$ and $L$ is a $\Delta(G)$-assignment for $H$, then all $L$-colorings of $H$
are $L$-equivalent.
\end{mainthm}

\subsection{Proof Outline}
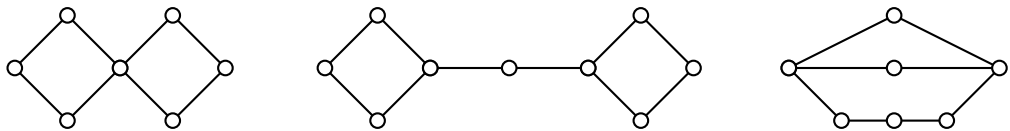
\begin{figure}[t]
\centering
\begin{tikzpicture}[thick, scale=.7]
\tikzstyle{uStyle}=[shape = circle, minimum size = 5.5pt, inner sep = 1pt,
outer sep = 0pt, draw, fill=white]
\tikzstyle{lStyle}=[shape = circle, draw=none, fill=none]
\tikzset{every node/.style=uStyle}
\draw[white] (0,-1.5) -- (.1,-1.5);

\draw (0,0) node {} 
--++ (1,1) node {}
--++ (1,-1) node {}
--++ (-1,-1) node {}
--++ (-1,1) 
--++ (-1,1) node {}
--++ (-1,-1) node {}
--++ (1,-1) node {}
--++ (1,1) node {};

\begin{scope}[xshift=3.5in]

\draw (0,0) node {} 
--++ (1,1) node {}
--++ (1,-1) node {}
--++ (-1,-1) node {}
--++ (-1,1) node {}
--++ (-1.5,0) node {}
--++ (-1.5,0) node {}
--++ (-1,1) node {}
--++ (-1,-1) node {}
--++ (1,-1) node {}
--++ (1,1) node {};
\end{scope}

\begin{scope}[xshift=5.0in]
\draw (0,0) node {} 
--++ (1,-1) node {}
--++ (1,0) node {}
--++ (1,0) node {}
--++ (1,1) node {}
--++ (-2,1) node {}
--++ (-2,-1) node {}
--++ (2,0) node {}
--++ (2,0);
\end{scope}
\end{tikzpicture}
\caption{Left: A short barbell.  Center: A non-short barbell.  Right: A
$\theta$-graph.\label{intro-pics}}
\end{figure}

In the rest of this introduction 
%(before ending with our definitions and terminology) 
we state our key lemmas and prove our main
results, assuming the lemmas.  In Section~\ref{swappability-sec} we prove the
lemmas on swappability.  And in Section~\ref{structural-sec} we use discharging
to prove our two structural lemmas.
A \Emph{barbell} is formed from two vertex-disjoint cycles by identifying one
vertex of each cycle with distinct ends of a path; see Figure~\ref{intro-pics}.  
We also allow that the path
has length 0, in which case we simply identify one vertex of the first cycle
with one vertex of the second; in this case, the barbell is
\EmphE{short}{-4mm}.  A \Emph{$\theta$-graph} is formed from two
non-adjacent vertices $v$ and $w$ by adding three internally disjoint $v,w$-paths.
A \EmphE{plane graph}{0mm} is a planar graph together with some planar embedding.
For a plane graph $G$, the subgraph \emph{$G_3$} is
induced by all edges incident to vertices of degree 3, and the subgraph
\emph{$G_2$} is induced by all edges incident to vertices of degree 2 that lie
on at least one 3-face.\aaside{$G_3$, $G_2$}{-4.75mm}
Most of our terminology and notation are standard.  But, for completeness, 
we provide a number of definitions at the end of the introduction.

Our first lemma was proved by Cranston and Mahmoud~\cite{CM}.  The proof is
long, so we omit it.\footnote{Since~\cite{CM} is not yet %available, 
published, we include a proof of Lemma~1 in the appendix.} 
However, in Section~\ref{swappability-sec} we do prove
Corollary~\ref{easy-cor}; see Lemma~\ref{degen-lem}.  This is the
special case when $H$ is an isolated vertex, which suggests some of the ideas
used to prove the more general case of Lemma~\ref{CM-lem}.
For a graph $G$ and $f:V(G)\to \Z^+$, we say $G$ is \EmphE{$f$-swappable}{0mm} if $G$ is $L$-swappable
whenever $L$ is a list assignment with $|L(v)|=f(v)$ for each vertex $v$.

%\setcounter{lemi}{0}
%\begin{NoHyper}
\begin{lem}
Fix a graph $G$ and a function $f:V(G)\to \Z^+$.  Let $H$ be an induced subgraph
of $G$ such that $G-H$ is $f$-swappable.  Let $f'(x):=f(x)-(d_G(x)-d_H(x))$ for
all $x\in V(H)$.  If $f'(x)\ge d_H(x)$ for all $x\in V(H)$ and $H$ is
both $f'$-choosable and $f'$-swappable, then $G$ is $f$-swappable.
%\label{H-lem}
\label{CM-lem}
\end{lem}
%\end{NoHyper}

Erd\H{o}s, Rubin, and Taylor~\cite{ERT} characterized the connected graphs that are
not degree-choosable; these are precisely Gallai trees.  So each connected graph
$H$ is $f'$-choosable, as in the previous lemma, unless both (a) $f'(v)=d_H(v)$
for all $v\in V(H)$ and (b) $H$ is a Gallai tree.

\begin{cor}
Fix a graph $G$ and a function $f:V(G)\to \Z^+$, and let $v$ be a vertex with
$f(v)>d(v)$.  If $G-v$ is $f$-swappable, then $G$ is $f$-swappable.
\label{easy-cor}
\end{cor}

To prove each of our main results, we assume that it is false and choose $G$ and $L$ to
be a counterexample minimizing $\|G\|$.  Let $k:=|L(v)|$ for all vertices $v$.
The point of Lemma~\ref{CM-lem} is that if some subgraph $H$ is $f'$-swappable
(and $f'$-choosable), then it cannot appear as an induced subgraph of
$G$.  If so, then we get that $G-H$ is $f$-swappable by minimality, and this
implies that $G$ is $f$-swappable, by Lemma~\ref{CM-lem}.  For a planar graph
$G$, Borodin showed that either (i) $G$ contains a so-called ``light edge'', one
incident to fewer than $k$ other edges or (ii) $G$ contains an even length cycle
with each edge incident to $k-2$ edges outside the cycle.
Since neither (i) nor (ii) appears in a minimal
counterexample (that is, both (i) and (ii) are \emph{reducible}), his theorems
follows.  We can handle (i) similarly, using Corollary~\ref{easy-cor}.
However, the reducibility of (ii) for Borodin's theorems relies
on the fact that even cycles are 2-choosable.  As we noted above, even cycles are
not 2-swappable.  Thus, to prove our main theorems, our structural lemmas must give
stronger conclusions.  So we replace (ii) with two larger types of subgraphs
(see (C2) and (C3)), each of which contains at least two instances of (ii).

\begin{struct-lem}
\label{structural-lem1}
If $G$ is a simple plane graph with $\delta\ge 2$, then either 
\begin{itemize}
\item[(C1)] $G$ contains an edge $vw$ with $d(v)+d(w)\le \max\{11,\DeltaG+2\}$; or
\item[(C2)] $G_3$ contains a bipartite barbell (possibly short); or
\item[(C3)] $G_3$ contains a bipartite $\theta$-graph other than $K_{2,3}$.
\end{itemize}
\end{struct-lem}

\begin{struct-lem}
\label{structural-lem2}
If $G$ is a simple plane graph with $\delta\ge 2$, then either 
\begin{itemize}
\item[(C1)] $G$ contains an edge $vw$ with $d(v)+d(w)\le 16$; or
\item[(C2)] $G_2$ contains a bipartite barbell (possibly short); or
\item[(C3)] $G_2$ contains either $K_{2,4}$ or a bipartite $\theta$-graph other than $K_{2,3}$.
\end{itemize}
\end{struct-lem}

When we use Structural Lemma~\ref{structural-lem1} to prove Main
Theorem~\ref{main-thm1}, it is straightforward to show that (C1) and (C2) are
reducible, given Lemma~\ref{CM-lem}.  The same is true when using
Structural Lemma~\ref{structural-lem2} to Prove Main Theorem~\ref{main-thm2}. 
However, handling (C3) is more involved.  For now, we simply
state the needed result, and use it to prove our two main theorems.  In
Section~\ref{swappability-sec}, we provide the proof.  

\begin{lem}
\label{reduc-lem}
A graph $G$ is degree-swappable whenever it is the line graph of either (1) a
bipartite barbell (possibly short) or (2) a bipartite $\theta$-graph other
than $K_{2,3}$ or (3) $K_{2,4}$.
\end{lem}

The following result immediately implies our Main Theorem~\ref{main-thm1}.
It is slightly more technical, so as to better facilitate
a proof by minimal counterexample.

\begin{thm}
Let $G$ be a simple plane graph, let $J$ be its line graph, and let
$k:=\max\{10,\DeltaG+1\}$.  If $L$ is a $k$-assignment for $J$, then
$J$ is $L$-swappable.
\label{main1-strong}
\end{thm}

\begin{proof}
We view $L$ as an edge assignment for $G$, and say that $G$ is
\emph{$L$-edge-swappable} to mean that $J$ is $L$-swappable.
Suppose the theorem is false; let $G$ and $L$ be a counterexample minimizing
$\|G\|$.  Since we are coloring edges, we assume that $\delta(G)\ge 1$.  
First suppose that $\delta(G)=1$, and let $e$ be an edge incident to a 1-vertex.
By minimality, $G-e$ is $L$-edge-swappable.  And $e$ is incident to fewer than
$|L(e)|$ other edges.  So $G$ is $L$-edge-swappable, by Corollary~\ref{easy-cor}.

Now assume instead that $\delta(G)\ge 2$.
By Structural Lemma~\ref{structural-lem1}, we know that either
(C1) $G$ contains an edge $vw$ with $d(v)+d(w)\le 
\max\{11,\DeltaG+2\}$; or
(C2) $G_3$ contains a bipartite barbell (possibly short); or
(C3) $G_3$ contains a bipartite $\theta$-graph other than $K_{2,3}$.
First suppose that (C1) holds.  By the minimality of $G$, we know that $G-vw$ is
$L$-edge-swappable.  Note that $|L(vw)|= \max\{10,\DeltaG+1\}$, but $vw$ has at most
$\max\{9,\DeltaG\}$ incident edges, i.e., $d_H(vw)<|L(vw)|$.  So, by
Corollary~\ref{easy-cor}, we know that $G$ is $L$-edge-swappable.

Now suppose instead that (C2) or (C3) holds, but (C1) does not.  Let $H$ be the induced subgraph of
$G$ in (C2) or (C3), and let $H'$ denote the line graph of $H$.  By definition, each
edge $vw$ in $G_3$ has $d(v)+d(w)\le 3+\DeltaG$.  Since (C1) does not hold, this
inequality holds with equality.  That is, the number of edges incident to $vw$
is exactly $3+\DeltaG-2=\DeltaG+1$; so $d_J(vw)=|L(vw)|$.  Thus,
$|L(vw)|-(d_J(vw)-d_{H'}(vw))=d_{H'}(vw)$.  By Lemma~\ref{reduc-lem}, we know
that $H'$ is degree-swappable.  Hence, Lemma~\ref{CM-lem} implies that $J$ is
$L$-swappable, i.e., that $G$ is $L$-edge-swappable, as desired.
\end{proof}

Our next result immediately implies Main Theorem~\ref{main-thm2}.  Similarly
to Theorem~\ref{main1-strong}, it is slightly more technical than Main
Theorem~\ref{main-thm2}, to better facilitate a proof by minimal counterexample.

\begin{thm}
Let $G$ be a simple plane graph, let $J$ be its line graph, and let
$k:=\max\{15,\DeltaG\}$.  If $L$ is a $k$-assignment for $J$, then
$J$ is $L$-swappable.
\label{main2-strong}
\end{thm}
We omit the proof of Theorem~\ref{main2-strong}, since it is essentially
identical to that of Theorem~\ref{main1-strong}.  The only differences are that
we use $G_2$ in place of $G_3$ and that we use Structural Lemma 2 in place of
Structural Lemma 1.

\subsection{Definitions and Terminology}
We write $[k]$ to denote $\{1,\ldots,k\}$.\aside{$[k]$}\aaside{$|G|$,
$\|G\|$}{4mm}
We write $d(v)$ for the degree of a vertex $v$.
For a graph $G$, we write $|G|$ to denote $|V(G)|$ and write $\|G\|$ to denote $|E(G)|$.  
We write $\delta(G)$ for the minimum degree in $G$ and
$\Delta(G)$ for the maximum degree; when $G$ is clear from context, we may simply
write $\delta$ or $\Delta$.  
A \EmphE{plane graph}{-4mm} is a planar graph
along with some planar embedding.  
A \EmphE{$k$-vertex}{4mm} is a vertex of degree $k$, and a
\EmphE{$k$-neighbor}{4mm} (of some vertex $v$) is an adjacent $k$-vertex.  
In a plane graph, a $k$-face is a face of length $k$.  
The \emph{length} of a face $f$ is denoted by $\ell(f)$.
A \EmphE{$k^+$-vertex}{8mm} (resp.~\emph{$k^-$-vertex}) is a vertex of degree
at least $k$  (resp.~at most $k$). We define $k^+$-neighbors, $k^-$-neighbors, $k^+$-faces,
and $k^-$-faces analogously.  The \emph{line graph} $J$ of a graph $G$ has $V(J)=E(G)$
and $e_1e_2\in E(J)$ if and only if $e_1$ and $e_2$ share an endpoint in $G$.
The \emph{Cartesian product}, \Emph{$G_1\dbox
G_2$}, of graphs $G_1$ and $G_2$ has vertex set $\{(v,w):v\in V(G_1),w\in
V(G_2)\}$ and $(v_1,w_1)$ is adjacent to $(v_2,w_2)$ if either (a) $v_1=v_2$
and $w_1w_2\in E(G_2)$ or (b) $w_1=w_2$ and $v_1v_2\in E(G_1)$.

\section{Swappability Lemmas}
\label{swappability-sec}
In this section, we prove all of our swappability lemmas.
Recall, for a graph $G$ and a list assignment $L$ for $G$, that $G$ is
\Emph{$L$-swappable} if every two of its $L$-colorings $\vph_1$ and $\vph_2$ are
$L$-equivalent.  That is, if $\vph_1$ can be transformed to $\vph_2$ by a
sequence of $L$-valid Kempe swaps.

Our first lemma is essentially due to Las Vergnas and Meyniel~\cite{lVM}.  They
proved it in the context of coloring, rather than list-coloring, but the proof
for list-coloring is nearly identical.  It is a slight reformulation of
Corollary~\ref{easy-cor} from the introduction.

\begin{lem}
\label{degen-lem}
Fix a connected graph $G$, a list-assignment $L$ for $G$, and $v\in V(G)$ with
$|L(v)|>d(v)$.  %Let $G':=G-v$ and let $L'$ denote $L$ restricted to $G'$.
If $G-v$ is $L$-swappable, then $G$ is $L$-swappable.
\end{lem}
\begin{proof}
Let $\vph_1$ and $\vph_2$ denote $L$-colorings of $G$, and let $\vph'_1$ and
$\vph'_2$ denote their restrictions to $G'$.  Since $G'$ is $L'$-swappable,
there exist $L'$-colorings $\psi'_0,\ldots,\psi'_t$ of $G'$ such that
$\psi'_0=\vph'_1$ and $\psi'_t=\vph'_2$ and $\psi'_i$ differs from
$\psi'_{i-1}$ by a single $L$-valid Kempe swap, for each $i\in [t]$.  Now we
use induction on $i$ to 
extend each $\psi'_i$ to an $L$-coloring $\psi_i$ of $G$ such that $\psi_i$ and
$\psi_{i-1}$ are $L$-equivalent.  (The base case, $i=0$, is trivial, since
$\psi_0'$ is a restriction of $\vph_1$.)

Suppose that $\psi'_i$ differs from
$\psi'_{i-1}$ by an $\alpha,\beta$-swap at $v_i$, for some $v_i\in
V(G)\setminus\{v\}$ and
some colors $\alpha$ and $\beta$.  If $\psi_{i-1}(v)\notin\{\alpha,\beta\}$ or
if $v$ is not in the same $\alpha,\beta$-component of $\psi_{i-1}$ as $v_i$,
then we form $\psi_i$ from $\psi_{i-1}$ by performing the same
$\alpha,\beta$-swap at $v_i$.  This approach also works if
$\{\alpha,\beta\}\subseteq
L(v)$ and $v$ is in the same $\alpha,\beta$-component as $v_i$, but $v$ has
degree 1 in that component.  So suppose $v$ is in the same
$\alpha,\beta$-component as $v_i$, but either $|L(v)\cap \{\alpha,\beta\}|=1$
or $v$ has degree at least 2 in that $\alpha,\beta$-component.  Since
$|L(v)|>d(v)$, there exists $\gamma\in L(v)$ that is unused by $\psi_{i-1}$
on the closed neighborhood of $v$.  We first recolor $v$ with $\gamma$, and
then perform the $\alpha,\beta$-swap at $v_i$.  %By induction on $i$, 
As desired, this yields an $L$-coloring $\psi_t$ that restricts to $\psi'_t$. 
Finally, if $\psi_t(v)\ne \vph_2(v)$, then we recolor $v$.
\end{proof}

Lemma~\ref{degen-lem} has a number of useful corollaries.

\begin{cor}
\label{degen-cor1}
Let $G$ be a connected graph and $L$ be a list-assignment for $G$.  If there
exists a vertex order $\sigma$ such that each vertex is preceded in $\sigma$ by
fewer than $|L(v)|$ of its neighbors, then $G$ is $L$-swappable.  In particular,
this is true if $|L(v)|\ge d(v)$ for all $v\in V(G)$ and there exists some
vertex $w$ such that $|L(w)|>d(w)$.
\end{cor}
\begin{proof}
We use induction on $|G|$.  The base case, $|G|=1$, is trivial; so assume that
$|G|\ge 2$.  Let $v$ denote the final vertex in $\sigma$, let $G':=G-v$, let
$L'$ denote $L$ restricted to $G'$, and let $\sigma'$ denote $\sigma$ with $v$
deleted.  By the induction hypothesis, $G'$ is $L'$-swappable.  So, by
Lemma~\ref{degen-lem}, $G$ is $L$-swappable.  This proves the first statement.
The second statement follows from the first, ordering by non-increasing distance
from $w$, since each vertex $x$ other than $w$ is followed in the order by its
neighbor on a shortest $x,w$-path.
\end{proof}

Recall that a \Emph{degree assignment} for a graph $G$ is a list assignment $L$
such that $|L(v)|=d(v)$ for all $v\in V(G)$.  A graph is
\EmphE{degree-swappable}{6mm}
if it is $L$-swappable for every degree assignment $L$.  
For a graph $G$ and list-assignment $L$, we typically let $\LL$ denote the set
of all $L$-colorings of $G$.  For some $\LL'\subseteq \LL$, we say that $\LL'$
\emph{mixes} if every two $L$-colorings in $\LL'$ are $L$-equivalent.
For each $v\in V(G)$ and
each $\alpha\in L(v)$, let \emph{$\LL_{v,\alpha}$}\aaside{mixes
$\LL_{v,\alpha}$}{0mm} denote the set of all $L$-colorings
$\vph$ with $\vph(v)=\alpha$.  The focus of this section is to prove that
various graphs are degree-swappable. 

\begin{cor}
\label{degen-cor2}
Let $L$ be a degree assignment for $G$.  Fix $v\in V(G)$ such that $G-v$ is
connected, and fix $\alpha\in L(v)$.  If there exists $w\in N(v)$ such that
$\alpha\notin L(w)$, then $\LL_{v,\alpha}$ mixes.
\end{cor}
\begin{proof}
We simply color $v$ with $\alpha$ and apply the second statement of
Corollary~\ref{degen-cor1} to $G-v$.
\end{proof}

\begin{cor}
\label{degen-cor3}
Let $L$ be a degree assignment for $G$.  Fix $v,w,x\in V(G)$ such that $v,x\in
N(w)$ and $vx\notin E(G)$ and $G-v-x$ is connected.  If there exists $\alpha\in
L(v)\cap L(x)$, then $\LL_{v,\alpha}\cap \LL_{x,\alpha}$ mixes.
\end{cor}
\begin{proof}
We simply color $v$ and $x$ with $\alpha$ and then apply
Corollary~\ref{degen-cor1} to $G-v-x$.
\end{proof}

\begin{lem}
Let $G$ be a graph with an edge $vw$ and a degree assignment $L$.  If $G-vw$ is
degree-choosable and connected and $|L(v)\cap L(w)|\le 1$, then $G$ is
$L$-swappable.
\label{big-intersection}
\end{lem}
\begin{proof}
Suppose that $L(v)\cap L(w)\subseteq\{\alpha\}$, for some color $\alpha$.
Form $L'$ from $L$ by removing $\alpha$ from $L(v)$.
Form $L''$ from $L$ by removing $\alpha$ from $L(w)$.
Let $\LL_1$ denote the set of all $L'$-colorings of
$G-vw$ and let $\LL_2$ denote the set of all $L''$-colorings of $G-vw$.
Note that $\LL_1$ mixes, by the second statement of Corollary~\ref{degen-cor1},
since $|L'(w)|>d_{G-vw}(w)$.
Similarly, $\LL_2$ mixes.  
Form $L'''$ from $L$ by removing $\alpha$ from both $L(v)$ and $L(w)$.
By hypothesis $G-vw$ is degree-choosable, so $G-vw$ has an $L'''$-coloring $\vph$.
But $\vph$ is both an $L'$-coloring and an $L''$-coloring.  So $\LL_1$ mixes
with $\LL_2$.  Note that every $L$-coloring is either an $\LL_1$-coloring or an
$\LL_2$-coloring (or both).  Thus, $G$ is $L$-swappable.
\end{proof}

To prove that a graph $G$ is $L$-swappable, we often follow the same approach as
we used to prove the previous lemma.  We let $\LL$ denote the set of all
$L$-colorings of $G$, and we find sets of $L$-colorings
$\LL_1,\ldots,\LL_t$ such that $\LL=\bigcup_{i=1}^t\LL_i$.  We show that each
set $\LL_i$ mixes.  Finally, for each $i>1$, we show there exists $j<i$ such that
$\LL_i\cap \LL_j\ne \emptyset$.  This proves that $\LL$ mixes, as desired.
Often, we choose $\LL_i$ to be $\LL_{v,\alpha}$ (for some $v$ and $\alpha\in
L(v)$) or some other set that we know mixes by
Corollaries~\ref{degen-cor1}--\ref{degen-cor3}. 

The goal of this section is to prove the following lemma, which we stated in the
introduction.  

\begin{NoHyper}
\setcounter{lemi}{2}
\begin{lemi}
A graph $G$ is degree-swappable whenever it is the line graph of either (1) a
bipartite barbell (possibly short) or (2) a bipartite $\theta$-graph other
than $K_{2,3}$ or (3) $K_{2,4}$.
\end{lemi}
\end{NoHyper}

We split the proof of this lemma into three parts.
In Lemma~\ref{barbell-swappable-lem}, we handle (1).
In Lemma~\ref{K24-lem}, we handle (3).
And in Lemmas~\ref{short-theta-lem} and~\ref{prism-lem}, we handle (2).

\begin{lem}
The line graph of every bipartite barbell is degree-swappable.
\label{barbell-swappable-lem}
\end{lem}
\begin{proof}
Consider a (non-short) barbell $G$ that consists of an even cycle
$v_1\cdots v_{2s}$, another even cycle $w_1\cdots w_{2t}$, and a
$v_1,w_1$-path $P$ (possibly of length 1).  Let $H$ be the line graph of $G$
and let $L$ be a degree assignment for $H$.  We will show that $H$ is
$L$-swappable.  Note that $d_H(v_1v_2)=d_H(v_{2s}v_{1})=
d_H(w_1w_2)=d_H(w_{2t}w_{1})=3$ and $d_H(x)=2$ for all other $x$ of the form
$v_iv_{i+1}$ or $w_jw_{j+1}$.
Denote $L(v_2v_3)$ by $\{a,b\}$.  By Lemma~\ref{big-intersection}, we assume
that $L(v_iv_{i+1})=\{a,b\}$ for all $i\in \{2,\ldots,2t-1\}$.  Further,
$L(v_1v_2)=\{a,b,c\}$ for some color $c$.  Note that $\LL_{v_1v_2,c}$ mixes,
by Corollary~\ref{degen-cor2}.  If $\vph(v_1v_2)=a$, then
$\vph(v_{2i}v_{2i+1})=b$ and $\vph(v_{2i+1}v_{2i+2})=a$ for all
$i\in[s-1]$.  We first color $v_1v_2$ and $v_{2s-1}v_{2s}$ with $a$, and next
apply Corollary~\ref{degen-cor3} to $H-\{v_2v_3,\ldots,v_{2s-2}v_{2s-1}\}$.
Thus,
%shows that 
$\LL_{v_1v_2,a}$ mixes.  Similarly, $\LL_{v_1v_2,b}$ mixes.
Denote $L(w_1w_2)$ by $\{a',b',c'\}$.  By symmetry, each of 
$\LL_{w_1w_2,a'}$, $\LL_{w_1w_2,b'}$, and $\LL_{w_1w_2,c'}$ mixes.

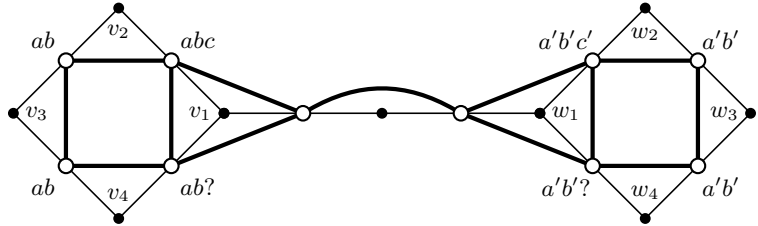
\begin{figure}[!h]
\centering
\begin{tikzpicture}[scale=1.4]
\tikzstyle{bStyle}=[shape = circle, minimum size = 3.5pt, inner sep = 1pt,
outer sep = 0pt, draw, fill=black]
\tikzstyle{wStyle}=[shape = circle, minimum size = 5.5pt, inner sep = 1pt,
outer sep = 0pt, draw, fill=white]
\tikzstyle{lStyle}=[shape = circle, draw=none, fill=none]
\tikzset{every node/.style=bStyle}
\draw[white] (0,-1.5) -- (.1,-1.5);

\draw[semithick] (0,0) node[label={[xshift=.03cm]right:\footnotesize{$w_1$}}] (w1) {}
--++ (1,1) node[label={[yshift=-.03cm]below:\footnotesize{$w_2$}}] (w2) {}
--++ (1,-1) node[label={[xshift=-.03cm]left:\footnotesize{$w_3$}}] (w3) {}
--++ (-1,-1) node[label={[yshift=.03cm]above:\footnotesize{$w_4$}}] (w4) {}
--++ (-1,1) node (w5) {}
--++ (-1.5,0) node (w6) {}
--++ (-1.5,0) node[label={[xshift=-.03cm]left:\footnotesize{$v_1$}}] (v1) {}
--++ (-1,1) node[label={[yshift=-.03cm]below:\footnotesize{$v_2$}}] (v2) {}
--++ (-1,-1) node[label={[xshift=.03cm]right:\footnotesize{$v_3$}}] (v3) {}
--++ (1,-1) node[label={[yshift=.03cm]above:\footnotesize{$v_4$}}] (v4) {}
--++ (1,1) node {};

\begin{scope}[thick]
\foreach \i/\j in {1/2, 2/3, 3/4, 4/1}
{
\draw node[wStyle] (x\i\j) at (barycentric cs:v\i=1,v\j=1) {};
\draw node[wStyle] (y\i\j) at (barycentric cs:w\i=1,w\j=1) {};
}

\draw node[wStyle] (z1) at (barycentric cs:v1=1,w6=1) {};
\draw node[wStyle] (z2) at (barycentric cs:w5=1,w6=1) {};

\foreach \i/\j in {x12/x23, x23/x34, x34/x41, x41/x12, x12/z1, x41/z1, 
y12/y23, y23/y34, y34/y41, y41/y12, y41/z2, y12/z2} %, y12/z2}
\draw[ultra thick] (\i) -- (\j);
\draw[ultra thick] (z1) edge[bend left] (z2);

\draw (x12) ++ (.25,.20) node[lStyle] {\footnotesize{$abc$}};
\draw (x23) ++ (-.22,.20) node[lStyle] {\footnotesize{$ab$}};
\draw (x34) ++ (-.22,-.20) node[lStyle] {\footnotesize{$ab$}};
\draw (x41) ++ (.25,-.20) node[lStyle] {\footnotesize{$ab?$}};

\draw (y12) ++ (-.25,.20) node[lStyle] {\footnotesize{$a'b'c'$}};
\draw (y23) ++ (.22,.20) node[lStyle] {\footnotesize{$a'b'$}};
\draw (y34) ++ (.22,-.20) node[lStyle] {\footnotesize{$a'b'$}};
\draw (y41) ++ (-.25,-.20) node[lStyle] {\footnotesize{$a'b'?$}};

\end{scope}
\end{tikzpicture}

\caption{A non-short barbell, its line graph, and (most of) a degree-assignment for the
line graph. }
\end{figure}

Finally, if
$\alpha\in L(v_1v_2)$, $\beta\in L(w_1w_2)$, %and $\alpha\ne\beta$, 
then $\LL_{v_1v_2,\alpha}\cap\LL_{w_1w_2,\beta}\ne \emptyset$, as follows.  If
$\alpha=c$, then we color $v_1v_2$ with $\alpha$, color $w_1w_2$ with $\beta$,
and color greedily in order of non-increasing distance from $v_2v_3$ (in the
uncolored subgraph), since
$\alpha\notin L(v_2v_3)$.  
%A similar approach works if $\beta=c'$.  
So instead assume that $\alpha\in\{a,b\}$. % and $\beta\in \{a',b'\}$.  
Now we color $v_1v_2,v_2v_3,\ldots,v_{2s-1}v_{2s}$, alternating $a$ and $b$, color $w_1w_2$
with $\beta$, and color greedily in order of non-increasing distance from
$v_{2s}v_1$.  So $\LL_{v_1v_2,\alpha}\cap\LL_{w_1w_2,\beta}\ne\emptyset$. 
Thus, $\LL$ mixes.

Now consider a short barbell.  The proof is similar.
The main differences are that $d_H(v_1v_2)=d_H(w_1w_2)=4$ and that $v_1v_2$ and
$w_1w_2$ are adjacent.  However, given distinct $\alpha_1,\alpha_2\in
L(v_1v_2)$,
there exists $\beta\in L(w_1w_2)\setminus\{\alpha_1,\alpha_2\}$.  Further,
$\LL_{v_1v_2,\alpha_1}\cap \LL_{w_1w_2,\beta}\ne \emptyset$ and
$\LL_{v_1v_2,\alpha_2}\cap \LL_{w_1w_2,\beta}\ne \emptyset$.  Thus,
$\bigcup_{\alpha\in L(v_1v_2)}\LL_{v_1v_2,\alpha}$ mixes.  But this includes all
of $\LL$.  So, $H$ is degree-swappable, as desired.
\end{proof}

Observe that the line graph of $K_{2,4}$ is $K_4\dbox K_2$.  Thus, the next
lemma proves (3) in Lemma~\ref{reduc-lem}.

\begin{lem}
The graph $K_4\dbox K_2$ is 4-swappable.
\label{K24-lem}
\end{lem}
\begin{proof}
Let $G:=K_4\dbox K_2$ and let $L$ be a 4-assignment for $G$.  We denote $V(G)$
by $\{v_1,v_2,v_3,v_4,$ $w_1,w_2,w_3,w_4\}$ such that $E(G)=\{v_iw_i$ for all
$i\in [4]\}\cup\{v_iv_j$ for all distinct $i,j\in [4]\}\cup \{w_iw_j$ for
all distinct $i,j\in [4]\}$; see Figure~\ref{K4K2-fig}.  
It is convenient to note that $K_4\dbox K_2$ is 4-connected; given distinct
$i,j\in[4]$, it is easy to find 4 internally-disjoint $v_i,w_j$-paths.
Thus, whenever we color at most 3 vertices, the uncolored subgraph is connected.

Suppose there exist distinct $i,j\in[4]$ such that $L(v_i)\ne L(w_j)$.  
We will show that $\LL$ mixes.  By symmetry, assume that $i=1$ and $j=2$.  Let
$A:=L(v_1)\cap L(w_2)$, let $A_1:=L(w_1)\cap A$, let $A_2:=L(v_2)\cap A$,  
let $B_1:=L(w_1)\setminus A$, and let $B_2:=L(v_2)\setminus A$.  So
$L(w_1)=A_1\cup B_1$ and $L(v_2)=A_2\cup B_2$. 
For each $\alpha\in B_1$, the set $\LL_{w_1,\alpha}$ mixes, by
Corollary~\ref{degen-cor2}, since $\alpha\notin A$.  Similarly, for each
$\alpha\in B_2$, the set $\LL_{v_2,\alpha}$ mixes, again since $\alpha\notin A$.  
Furthermore, for every $\alpha_1\in B_1$ and $\alpha_2\in B_2$, there exists an
$L$-coloring $\vph$ with $\vph(w_1)=\alpha_1$ and $\vph(v_2)=\alpha_2$, since
$G-w_1-v_2$ is degree-choosable.  
Let $\CC_1:= \left(\bigcup_{\alpha_1\in B_1}\LL_{w_1,\alpha_1} \right)\cup
\left(\bigcup_{\alpha_2\in B_2}\LL_{v_2,\alpha_2}\right)$, and note that $\CC_1$
mixes.  For each $\alpha\in A_1\cap A_2$, the set $\LL_{w_1,\alpha}\cap\LL_{v_2,\alpha}$
mixes, by Corollary~\ref{degen-cor3} (however, possibly $A_1\cap A_2=\emptyset$).
Since $|L(v_1)\cup L(w_2)|\ge 5 > (|L(w_1)|+|L(v_2)|)/2$, there exists $\beta\in
L(v_1)\cup L(w_2)\setminus (L(v_2)\cap L(w_1))$.  By symmetry, assume that
$\beta\in L(v_1)$.  Now $\LL_{v_1,\beta}$ mixes.  Further, for each $\alpha\in
A_1\cap A_2$, 
note that $\LL_{w_1,\alpha}\cap\LL_{v_2,\alpha}\cap\LL_{v_1,\beta}\ne\emptyset$.
Let $\CC_2:=\bigcup_{\alpha\in A_1\cap A_2}(\LL_{w_1,\alpha}\cap
\LL_{v_2,\alpha})$ and note that $\CC_2$ mixes (again, if $A_1\cap
A_2=\emptyset$, then $\CC_2=\emptyset$, but the statements above and below
still hold, trivially).  Since $\beta\notin L(v_2)\cap L(w_1)$, there exists
$\gamma\in (B_1\cup B_2)\setminus \{\beta\}$.  If $\gamma\in B_1$, then 
$\LL_{v_1,\beta}\cap\LL_{w_1,\gamma}\ne\emptyset$ (since $G-v_1-w_1$ is
degree-choosable).  And if $\gamma\in B_2$, then 
$\LL_{v_1,\beta}\cap\LL_{v_2,\gamma}\ne\emptyset$.
Thus, $\CC_1\cup \CC_2$ mixes.

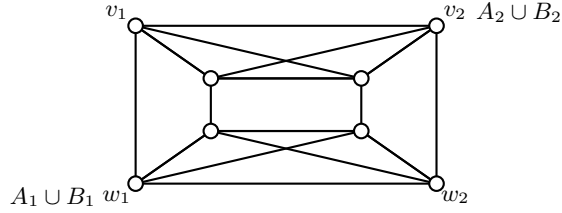
\begin{figure}[!h]
\centering
\begin{tikzpicture}[thick, scale=1.0, yscale=.7]
\tikzstyle{bStyle}=[shape = circle, minimum size = 3.5pt, inner sep = 1pt,
outer sep = 0pt, draw, fill=black]
\tikzstyle{wStyle}=[shape = circle, minimum size = 5.5pt, inner sep = 1pt,
outer sep = 0pt, draw, fill=white]
\tikzstyle{lStyle}=[shape = circle, draw=none, fill=none]
\tikzset{every node/.style=wStyle}

\draw (-1,1) node (v1) {} --++ (1,-1) node (v3) {} --++ (2,0) node (v4) {} --++ (1,1) node (v2) {}
(v1) ++ (0,-3) node (w1) {} --++ (1,1) node (w3) {} --++ (2,0) node (w4) {} --++ (1,-1) node (w2) {};

\foreach \i/\j in {1/2, 1/3, 1/4, 2/3, 2/4, 3/4}
\draw (v\i) -- (v\j) (w\i) -- (w\j);

\foreach \i in {1,2,3,4}
\draw (v\i) -- (w\i);

\draw (v1) ++ (-.25, .25) node[lStyle] {\footnotesize{$v_1$}};
\draw (v2) ++ (.25, .25) node[lStyle] {\footnotesize{$v_2$}};
\draw (v2) ++ (1.1, .25) node[lStyle] {\footnotesize{$A_2\cup B_2$}};
\draw (w1) ++ (-.25, -.25) node[lStyle] {\footnotesize{$w_1$}};
\draw (w1) ++ (-1.1, -.25) node[lStyle] {\footnotesize{$A_1\cup B_1$}};
\draw (w2) ++ (.25, -.25) node[lStyle] {\footnotesize{$w_2$}};
\end{tikzpicture}
\caption{$K_4\dbox K_2$, along with a partition of $L(v_2)$ and
$L(w_1)$.\label{K4K2-fig}}
\end{figure}

Finally, consider a coloring $\vph$ with $\vph(w_1),\vph(v_2)\in A$ and
$\vph(w_1)\ne \vph(v_2)$.  Let $\alpha_1:=\vph(w_1)$ and $\alpha_2:=\vph(v_2)$.
First suppose there exists $i\in\{3,4\}$ such that $\vph(v_i)=\alpha_1$; by
symmetry, assume that $\vph(v_3)=\alpha_1$.  Now $\LL_{w_1,\alpha_1}
\cap\LL_{v_3,\alpha_1}$ mixes and, for each $\beta\in B_2$, there exists
$\vph'\in\LL_{w_1,\alpha_1}\cap\LL_{v_3,\alpha_1}\cap\LL_{v_2,\beta}$.  Since
$\vph'\in\CC_1$, we are done.  Otherwise, there exists $i\in\{3,4\}$ such that
$\vph(v_i)\notin L(v_1)$; by symmetry, assume $i=3$, and let $\beta:=\vph(v_3)$.
Now $\LL_{v_3,\beta}$ mixes, and for each $\gamma\in B_1$, we have
$\LL_{v_3,\beta}\cap\LL_{w_1,\gamma}\ne \emptyset$.  Since $\LL_{w_1,\gamma}$
mixes and is contained in $\CC_1$, we are done.
This concludes the case that $L(v_i)\ne L(w_j)$ for some distinct $i,j\in[4]$.
So we have $L(v_1)=L(w_2)=L(v_4)=L(w_1)=L(v_3)=L(w_4)=L(v_2)=L(w_3)$.
That is, all lists are identical.  Let $A:=L(v_1)$. 

Now we show that for each distinct pair $i,j\in[4]$ the set of all $L$-colorings
with $\vph(v_i)=\vph(w_j)$ mixes, and thus the set of all $L$-colorings mixes.  More
precisely, fix $\alpha\in A$.  Clearly $\LL_{v_i,\alpha}\cap\LL_{w_j,\alpha}$
mixes.  Also, fix distinct $k,\ell\in[4]$ such that $k\ne i$ and $\ell\ne j$,
and fix $\beta\in A\setminus\{\alpha\}$.  Now there exists an
$L$-coloring $\vph$ with $\vph(v_i)=\vph(w_j)=\alpha$ and
$\vph(v_k)=\vph(w_{\ell})=\beta$.  This implies that 
$\bigcup_{\alpha\in A}(\LL_{v_i,\alpha}\cap\LL_{w_j,\alpha})$ mixes; 
so let $D_{i,j}:=\bigcup_{\alpha\in A}(\LL_{v_i,\alpha}\cap\LL_{w_j,\alpha})$.
Finally, $D_{1,2}$ mixes with $D_{2,3}$, which mixes with $D_{1,4}$, which mixes
with $D_{3,2}$, which mixes with $D_{1,3}$.  Now we are done, since every
$L$-coloring of $G$ lies in $D_{1,2}\cup D_{1,3}\cup D_{1,4}$.  That is, 
$\LL$ mixes.
\end{proof}

\begin{lem}
If $G$ is a bipartite $\theta$-graph, with one of its three edge-disjoint paths
of length 1, then the line graph of $G$ is degree-swappable.
\label{short-theta-lem}
\end{lem}
\begin{proof}
Let $H$ be the line graph of $G$.  Now $H$ has four 3-vertices;
call these $v_1, v_2, w_1, w_2$, so that $v_1w_1,v_2w_2\in E(H)$; see
Figure~\ref{theta-line-fig}.  Let
$x_1,\ldots$ denote the internal 2-vertices of a $v_1,v_2$-path of even length, and let
$y_1,\ldots$ denote the internal 2-vertices of a $w_1,w_2$-path of even length.  Denote the
sole 4-vertex by $z$, where $N_H(z)=\{v_1,v_2,w_1,w_2\}$.  Denote $L(x_1)$ by
$\{a,b\}$.  By Lemma~\ref{big-intersection}, we assume that $L(x_i)=\{a,b\}$ for
all $i$, that $L(v_1)=\{a,b,c_1\}$, and that $L(v_2)=\{a,b,c_2\}$, for some
colors $c_1$ and $c_2$.  
Let $\LL_1:=\LL_{v_1,a}\cap \LL_{v_2,a}$ and $\LL_2:=\LL_{v_1,b}\cap
\LL_{v_2,b}$.  By Corollary~\ref{degen-cor3}, both $\LL_1$ and $\LL_2$ mix.
Note that $\LL_{v_1,c_1}$ mixes by Corollary~\ref{degen-cor2}, as does $\LL_{v_2,c_2}$.
Let $\LL_3:=\LL_{v_1,c_1}\cup \LL_{v_2,c_2}$.  If we color $v_1$ with $c_1$ and
color $v_2$ with $c_2$, then we can extend to an $L$-coloring, since the even
cycle induced by $w_1,w_2,z$, and the $y_i$'s is 2-choosable.  Thus, $\LL_3$
mixes.  Note also that $\LL_1, \LL_2, \LL_3$ partition $\LL$.

By symmetry, between the $v_i$'s and $w_i$'s and the $x_i$'s and $y_i$'s, we
assume there exist colors $a',b',c_1',c_2'$ such that $L(y_i)=\{a',b'\}$,
$L(w_1)=\{a',b',c_1'\}$, and 
$L(w_2)=\{a',b',c_2'\}$.  Analogous to $\LL_1,\LL_2,\LL_3$, we define
$\LL_4,\LL_5,\LL_6$, which each mix and which also partition $\LL$.
Since there exists an $L$-coloring in $\LL_3\cap \LL_6$, we note that $\LL_3\cup
\LL_6$ mixes.  By symmetry between $a$ and $b$, we assume that $c_1'\ne a$.
So there exists some $L$-coloring in $\LL_1\cap \LL_6$.  Similarly, there exists
some $L$-coloring in $\LL_3\cap\LL_4$.  So $\LL_1\cup \LL_3\cup\LL_4\cup \LL_6$
mixes.  Finally, either $b\ne a'$ or $b\ne c_1'$.  In the first case, there
exists a coloring in $\LL_2\cap \LL_4$; in the second, there exists a coloring
in $\LL_2\cap \LL_6$.  So $\LL_2$ mixes with $\LL_1\cup \LL_3\cup\LL_4\cup
\LL_6$.  Since $\LL_1$, $\LL_2$, $\LL_3$ partitions $\LL$, we are done.
\end{proof}

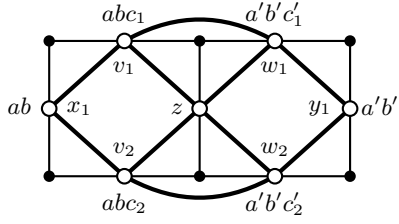
\begin{figure}[!t]
\centering
\begin{tikzpicture}[thick, yscale=.9]
\tikzstyle{bStyle}=[shape = circle, minimum size = 3.5pt, inner sep = 1pt,
outer sep = 0pt, draw, fill=black]
\tikzstyle{wStyle}=[shape = circle, minimum size = 5.5pt, inner sep = 1pt,
outer sep = 0pt, draw, fill=white]
\tikzstyle{bStyle}=[shape = circle, minimum size = 3.5pt, inner sep = 1pt,
outer sep = 0pt, draw, fill=black]
\tikzstyle{lStyle}=[shape = circle, draw=none, fill=none]
\tikzset{every node/.style=wStyle}

\begin{scope}
\tikzset{every node/.style=bStyle}
\draw (0,0) node {} -- (2,0) node {} -- (4,0) node {} -- (4,-2) node {} --
(2,-2) node {} -- (0,-2) node {} -- cycle (2,0) -- (2,-2);
\end{scope}

\draw[ultra thick] (1,0) node[thick] (v1) {} 
-- (2,-1) node[thick] (z) {}
-- (1,-2) node[thick] (v2) {}
-- (0,-1) node[thick] (x1) {}
-- (v1) (z) -- (3,0) node[thick] (w1) {}
-- (4,-1) node[thick] (y1) {}
-- (3,-2) node[thick] (w2) {}
-- (z) (v1) edge[bend left] (w1) (v2) edge[bend right] (w2);

\draw (v1) ++ (0,-.4) node[lStyle] {\footnotesize{$v_1$}};
\draw (v1) ++ (0,.4) node[lStyle] {\footnotesize{$abc_1$}};
\draw (v2) ++ (0,.4) node[lStyle] {\footnotesize{$v_2$}};
\draw (v2) ++ (0,-.4) node[lStyle] {\footnotesize{$abc_2$}};
\draw (x1) ++ (.4,0) node[lStyle] {\footnotesize{$x_1$}};
\draw (x1) ++ (-.4,.04) node[lStyle] {\footnotesize{$ab$}};
\draw (z) ++ (-.3,0) node[lStyle] {\footnotesize{$z$}};
\draw (w1) ++ (0,-.4) node[lStyle] {\footnotesize{$w_1$}};
\draw (w1) ++ (0,.4) node[lStyle] {\footnotesize{$a'b'c'_1$}};
\draw (w2) ++ (0,.4) node[lStyle] {\footnotesize{$w_2$}};
\draw (w2) ++ (0,-.4) node[lStyle] {\footnotesize{$a'b'c'_2$}};
\draw (y1) ++ (-.4,0) node[lStyle] {\footnotesize{$y_1$}};
\draw (y1) ++ (.4,.04) node[lStyle] {\footnotesize{$a'b'$}};
\end{tikzpicture}
\caption{A $\theta$-graph with a path of length 1, its line graph,
and (most of) a list assignment for the line graph.\label{theta-line-fig}}
\end{figure}

\begin{defn}
A \emph{prism} consists of three vertex disjoint paths $v_1\cdots w_1$,
$v_2\cdots w_2$, and $v_3\cdots w_3$, each of length at least 1, such that
$\{v_1,v_2,v_3\}$ and $\{w_1,w_2,w_3\}$ each induce $K_3$ and all interior
vertices of the paths have degree 2.  We denote the $v_i,w_i$-path by $P_i$ for
each $i\in [3]$.  See Figure~\ref{prism-fig}.
\label{prism-defn}
\end{defn}

If each of three edge-disjoint paths in a $\theta$-graph $G$ has length at
least 2, then the line graph of $G$ is a prism.  Thus,
Lemmas~\ref{short-theta-lem} and~\ref{prism-lem} imply (2) in
Lemma~\ref{reduc-lem}.

\begin{lem}
If $G$ is a prism other than $K_3\dbox K_2$, then $G$ is degree-swappable.
\label{prism-lem}
\end{lem}
\begin{proof}
Fix a degree-assignment $L$ for $G$.  Note, for each $st\in E(G)$, that $G-st$
is connected and contains a $\theta$-graph; thus it is degree-choosable.  So,
by Lemma~\ref{big-intersection}, we assume $|L(s)\cap L(t)|\ge 2$ for all
$st\in E(G)$.  We label the vertices as in Definition~\ref{prism-defn};
see Figure~\ref{prism-fig}.  We show that we can assume that
$L(v_1)=L(v_2)=L(v_3)$ (and similarly, that $L(w_1)=L(w_2)=L(w_3)$).
Suppose the contrary.

First suppose that $|L(v_1)\cap L(v_2)\cap L(v_3)|=2$.  We assume that
$L(v_1)=\{a,b,c_1\}$, $L(v_2)=\{a,b,c_2\}$, and $L(v_3)=\{a,b,c_3\}$ where
possibly $c_1=c_2$, but $c_1\ne c_3$ and $c_2\ne c_3$; 
see the center of Figure~\ref{prism-fig}. 
Now $\LL_{v_1,c_1}$ mixes, by Corollary~\ref{degen-cor2}, since
$v_1\in N(v_3)$, but $c_1\notin L(v_3)$.  Similarly, $\LL_{v_2,c_2}$ mixes and
$\LL_{v_3,c_3}$ mixes.  Further, $\LL_{v_1,c_1}\cap \LL_{v_3,c_3}\ne \emptyset$.
To see this, first color $v_1$ and $v_3$ as
prescribed and then color greedily in order of non-increasing distance from
$v_2$; we can color $v_2$ last since $v_3\in N(v_2)$, but $\vph(v_3)\notin
L(v_2)$.  Similarly, $\LL_{v_2,c_2}\cap\LL_{v_3,c_3}\ne \emptyset$.
Thus $\LL_{v_1,c_1}\cup \LL_{v_3,c_3}$ mixes and $\LL_{v_2,c_2}\cup
\LL_{v_3,c_3}$ mixes.  Since
$\LL_{v_1,c_1}\cup\LL_{v_2,c_2}\cup\LL_{v_3,c_3}=\LL$, we conclude that $\LL$
mixes when $|L(v_1)\cap L(v_2)\cap L(v_3)|=2$.  

Now suppose instead that $|L(v_1)\cap L(v_2)\cap L(v_3)|=1$.  By possibly
renaming colors we assume that $L(v_1)=\{a,b_2,b_3\}$,
$L(v_2)=\{a,b_1,b_3\}$, and $L(v_3)=\{a,b_1,b_2\}$, where 
$b_1,b_2,b_3$ are distinct; again, see Figure~\ref{prism-fig}. 
Now $\LL_{v_i,b_j}$ mixes by Corollary~\ref{degen-cor2}, 
whenever $i\in[3]$ and $b_j\in L(v_i)$.
Arguments similar to those above show that 
$\LL_{v_1,b_2}\cup \LL_{v_2,b_3}\cup \LL_{v_3,b_1}$ mixes.
Similarly, 
$\LL_{v_1,b_3}\cup \LL_{v_2,b_1}\cup \LL_{v_3,b_2}$ mixes.
So $\LL$ forms at most 2 equivalence classes.  We must show that it forms only
a single equivalence class.  Since $G\ne K_3\dbox K_2$, some vertex in
$\{v_1,v_2,v_3\}$ has a 2-neighbor.  By symmetry, we assume that it is $v_1$,
and we call its 2-neighbor $x$.  Pick $\alpha\in L(v_1)\setminus L(x)$. 
Now $\LL_{v_1,\alpha}$ mixes.  Further, $\LL_{v_1,\alpha}\cap
\LL_{v_2,b_1}\ne\emptyset$ and $\LL_{v_1,\alpha}\cap
\LL_{v_3,b_1}\ne\emptyset$.  Thus, $\LL$ mixes when $|L(v_1)\cap L(v_2)\cap
L(v_3)|=1$.  Hence, as desired, we conclude that $|L(v_1)\cap L(v_2)\cap
L(v_3)|=3$, i.e., $L(v_1)=L(v_2)=L(v_3)$.  By symmetry, we also have
$L(w_1)=L(w_2)=L(w_3)$.

Suppose that $L(v_1)\ne L(w_1)$.  Since $G$ is a prism other than $K_3\dbox K_2$,
some $v_i$ has a 2-neighbor.  By symmetry among $v_1,v_2,v_3$, we assume that
$v_1$ has a 2-neighbor.  For each 2-vertex $x$ on the path from $v_1$ to
$w_1$, we assume, by Lemma~\ref{big-intersection}, that $L(x)\subseteq L(v_1)$
and $L(x)\subseteq L(w_1)$.  Thus,
$L(v_1)\cap L(w_1)=L(x)$.  By symmetry, we assume that $L(v_1)=\{a,b,c\}$ and
$L(w_1)=\{a,b,d\}$.  Fix $i,j\in[3]$.  By Corollary~\ref{degen-cor2}, each of 
$\LL_{v_i,c}$ and $\LL_{w_j,d}$ mixes.  Further, $\LL_{v_i,c}\cup \LL_{w_j,d}$
mixes, as follows.  If $i\ne j$, then we construct an
$L$-coloring $\vph$ with $\vph(v_i)=c$ and $\vph(w_j)=d$, by coloring $v_i$ and
$w_j$ as desired, and then coloring greedily in order of non-increasing distance
from $v_j$.  Suppose instead that $i=j$.  Now there exist $L$-colorings
$\vph_1$, $\vph_2$, and $\vph_3$ with $\vph_1(v_i)=c$ and $\vph_1(w_{i+1})=d$,
with $\vph_2(w_{i+1})=d$ and $\vph_2(v_{i+2})=c$, and with $\vph_3(v_{i+2})=c$
and $\vph_3(w_i)=d$ (with all subscripts modulo 3).  Thus,  $\LL_{v_i,c}\cup
\LL_{w_i,d}$ mixes.  But every $L$-coloring $\vph$
has $\vph(v_i)=c$ for some $i$ and has $\vph(w_j)=d$ for some $j$.  Thus, 
$\LL$ mixes when $L(v_1)\ne L(w_1)$.  So we assume that $L(v_1)=L(w_1)$.  By the
previous paragraph, this implies that $L(s)=\{a,b,c\}$ for every 3-vertex $s$.

Now we consider three cases, depending on the number of $i\in[3]$ such that the
$v_i,w_i$-path $P_i$ has length longer than 1.  By symmetry, we assume that
$\ell(P_1)\ge \ell(P_2)\ge \ell(P_3)$; so $\ell(P_1)>1$, since $G$ is a prism
other than $K_2\dbox K_3$.  Recall that all interior vertices on each $P_i$
have the same list, that this list is contained in $L(v_i)$, and that
$L(w_i)=L(v_i)$.

\begin{figure}[t]
\centering
\begin{tikzpicture}[thick, scale=.90, xscale=.7]
\tikzstyle{bStyle}=[shape = circle, minimum size = 3.5pt, inner sep = 1pt,
outer sep = 0pt, draw, fill=black]
\tikzstyle{wStyle}=[shape = circle, minimum size = 5.5pt, inner sep = 1pt,
outer sep = 0pt, draw, fill=white]
\tikzstyle{bStyle}=[shape = circle, minimum size = 3.5pt, inner sep = 1pt,
outer sep = 0pt, draw, fill=black]
\tikzstyle{lStyle}=[shape = circle, draw=none, fill=none]
\tikzset{every node/.style=wStyle}
\def\rad{.5in}

\draw (120:\rad) node (v1) {};
\draw (240:\rad) node (v2) {};
\draw   (0:\rad) node (v3) {};
\draw (v1) -- (v2) -- (v3) -- (v1);
\draw (v1) ++ (.3,-.45) node[lStyle] {\footnotesize{$v_1$}};
\draw (v2) ++ (.3,.45) node[lStyle] {\footnotesize{$v_2$}};
\draw (v3) ++ (-.55,0) node[lStyle] {\footnotesize{$v_3$}};

\begin{scope}[xshift=2.0in]
\draw  (60:\rad) node (w1) {};
\draw (300:\rad) node (w2) {};
\draw (180:\rad) node (w3) {};
\draw (w1) -- (w2) -- (w3) -- (w1);
\end{scope}
\draw (w1) ++ (-.3,-.45) node[lStyle] {\footnotesize{$w_1$}};
\draw (w2) ++ (-.3,.45) node[lStyle] {\footnotesize{$w_2$}};
\draw (w3) ++ (.55,0) node[lStyle] {\footnotesize{$w_3$}};

\draw[thick,snake=coil, segment aspect=0, segment amplitude=2pt,segment length=15pt] 
(v1) -- (w1) (v2) -- (w2) (v3) -- (w3);

\begin{scope}[xshift=4.40in]
\draw (120:\rad) node (v1) {};
\draw (240:\rad) node (v2) {};
\draw   (0:\rad) node (v3) {};
\draw (v1) -- (v2) -- (v3) -- (v1);
\draw (v1) ++ (.3,-.45) node[lStyle] {\footnotesize{$v_1$}};
\draw (v1) ++ (-.3,.4) node[lStyle] {\footnotesize{$abc_1$}};
\draw (v2) ++ (.3,.45) node[lStyle] {\footnotesize{$v_2$}};
\draw (v2) ++ (-.3,-.4) node[lStyle] {\footnotesize{$abc_2$}};
\draw (v3) ++ (-.55,0) node[lStyle] {\footnotesize{$v_3$}};
\draw (v3) ++ (.3,.4) node[lStyle] {\footnotesize{$abc_3$}};
\end{scope}

\begin{scope}[xshift=6.70in]
\draw (120:\rad) node (v1) {};
\draw (240:\rad) node (v2) {};
\draw   (0:\rad) node (v3) {};
\draw (v1) -- (v2) -- (v3) -- (v1);
\draw (v1) ++ (.3,-.45) node[lStyle] {\footnotesize{$v_1$}};
\draw (v1) ++ (-.3,.4) node[lStyle] {\footnotesize{$ab_2b_3$}};
\draw (v2) ++ (.3,.45) node[lStyle] {\footnotesize{$v_2$}};
\draw (v2) ++ (-.3,-.4) node[lStyle] {\footnotesize{$ab_1b_3$}};
\draw (v3) ++ (-.55,0) node[lStyle] {\footnotesize{$v_3$}};
\draw (v3) ++ (.3,.4) node[lStyle] {\footnotesize{$ab_1b_2$}};
\end{scope}
\end{tikzpicture}

\caption{Left: A prism $G$.  Center and right: List assignments for
$v_1$, $v_2$, $v_3$ such that 
$|L(v_1)\cap L(v_2)\cap L(v_3)|=2$ and
$|L(v_1)\cap L(v_2)\cap L(v_3)|=1$, respectively.
\label{prism-fig}}

\end{figure}
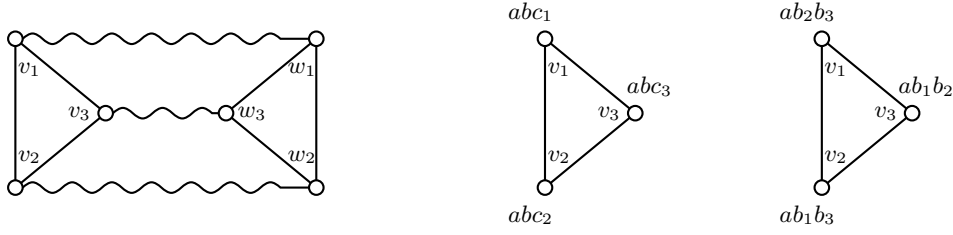

\textbf{Case 1: $\bm{\ell(P_1) > \ell(P_2)=\ell(P_3)=1}.$}
By symmetry, we assume that $L(x)=\{a,b\}$ for each interior vertex $x$ of
$P_1$.  Now $\LL_{v_1,c}$ and $\LL_{w_1,c}$ each mix.  Further,
$\LL_{v_1,c}\cap\LL_{w_1,c}\ne\emptyset$, so $\LL_{v_1,c}\cup\LL_{w_1,c}$ mixes.
Suppose
instead, by symmetry, that $\vph(v_1)=a$ and $\vph(w_1)\ne c$.  Performing
an $a,c$-swap at $v_1$ yields an $L$-coloring $\vph'$ with $\vph'(v_1)=c$. 
This is because $\{a,c\}\subseteq L(s)$ for every 3-vertex $s$ and
either $\vph(w_1)=b$ or $\vph(x)=b$, where $x$ is the neighbor of $w_1$ on
$P_1$.  So $\LL$ mixes.

\textbf{Case 2: $\bm{\ell(P_1)\ge \ell(P_2) > \ell(P_3)=1.}$}
By symmetry, we assume that $L(x)=\{a,b\}$ for each interior vertex $x$ of
$P_1$.  Suppose also that $L(y)=\{a,b\}$ for each interior vertex $y$ of $P_2$.
Now $\LL_{s,c}$ mixes for each $s\in\{v_1,v_2,w_1,w_2\}$.  Further,
$\LL_{v_1,c}\cap \LL_{w_2,c}\ne\emptyset$ and 
$\LL_{w_1,c}\cap \LL_{v_2,c}\ne\emptyset$, so 
$\LL_{v_1,c}\cup \LL_{w_2,c}$ and $\LL_{w_1,c}\cup \LL_{v_2,c}$ each
mix.  Note that the union of these two sets is $\LL$.
If $\ell(P_1)$ is odd, then $\LL_{v_2,c}\cap \LL_{w_2,c}\ne\emptyset$, so $\LL$
mixes, since the cycle in $G\setminus\{v_2,w_2\}$ is even and thus 2-choosable.
So we assume $\ell(P_1)$ is even. Similarly, we assume $\ell(P_2)$ is even.
Suppose that $\vph(v_3)=c$, $\vph(w_2)=c$ and $\vph(w_3)=\alpha$, for some
$\alpha\in\{a,b\}$.  Since $\ell(P_1)$ and $\ell(P_2)$ are both even, the
$\alpha,c$-component containing $w_3$ has vertex set $\{v_2,v_3,w_2,w_3\}$.
Now performing an $\alpha,c$-swap on this component yields a
valid $L$-coloring $\vph'$ with $\vph'(v_2)=c$. 
Thus, $\LL$ mixes when $L(y)=\{a,b\}$.

Instead assume that $L(y)=\{a,c\}$ for each interior vertex $y$ of $P_2$.  
Again $\LL_{v_1,c}$, $\LL_{w_1,c}$ $\LL_{v_2,b}$, and $\LL_{w_2,b}$ each mix. 
Further, $\LL_{v_1,c}\cup\LL_{w_1,c}\cup\LL_{v_2,b}\cup\LL_{w_2,b}$ mixes.  
If $\vph\in\LL_{v_1,a}\setminus\LL_{w_1,c}$, then an $a,c$-swap at $v_1$ yields
some $\vph'\in\LL_{v_1,c}$, as in Case 1, and we are done.  Similarly, if
$\vph\in\LL_{v_2,a}\setminus\LL_{w_2,b}$, then we use an $a,b$-swap at $v_2$. 
So instead we assume that $\vph\in(\LL_{v_1,b} \cap \LL_{v_2,c})
\setminus (\LL_{w_1,c} \cup \LL_{w_2,b})$.  But now a $b,c$-swap at $v_1,v_2$
yields an $L$-coloring in the set above that mixes.  So $\LL$ mixes.

\textbf{Case 3: $\bm{\ell(P_1)\ge \ell(P_2) \ge \ell(P_3)>1.}$}
By symmetry, we assume that $L(x)=\{a,b\}$ for all interior vertices $x$ of $P_1$.
First suppose that $L(y)=\{a,b\}$ and $L(z)=\{a,b\}$ for all interior vertices
$y$ of $P_2$ and all interior vertices $z$ of $P_3$.  Now $\LL_{v_i,c}$ and
$\LL_{w_j,c}$ each mix, for all $i,j\in[3]$.  Further, since
$\LL_{v_i,c}\cap\LL_{w_j,c}\ne\emptyset$ for all distinct $i,j\in[3]$, we get that
$\LL_{v_i,c}\cup\LL_{w_j,c}$ mixes for all distinct $i,j\in[3]$.  Applying this
idea repeatedly, $\LL_{v_1,c}$ mixes with $\LL_{w_2,c}$, which
mixes with $\LL_{v_3,c}$, which mixes with $\LL_{w_1,c}$, which mixes with
$\LL_{v_2,c}$.  Thus, $\cup_{i\in[3]}\LL_{v_i,c}$ mixes.  But this union is
$\LL$, so $\LL$ mixes.  

Suppose instead that $\vph(x)=\vph(y)=\{a,b\}$ for all interior vertices $x$ of
$P_1$ and interior vertices $y$ of $P_2$, but that $\vph(z)=\{a,c\}$ for all
interior vertices $z$ of $P_3$.  As above, $\LL_{v_1,c}$, $\LL_{v_2,c}$,
$\LL_{w_1,c}$, and $\LL_{w_2,c}$ each mix.  Also $\LL_{v_1,c}\cap\LL_{w_2,c}
\ne \emptyset$ and $\LL_{v_2,c}\cap\LL_{w_1,c}\ne\emptyset$.
Finally, $\LL_{v_1,c}\cap\LL_{w_1,c}\ne\emptyset$, since every cycle is
colorable from every 2-assignment unless all lists are equal (and the cycle is
odd).  Assume instead that $\vph(v_3)=\vph(w_3)=c$.  By symmetry, assume that
$\vph(v_1)=b$.  Now a $b,c$-swap at $v_3$ gives an $L$-coloring $\vph'$ with
$\vph'(v_1)=c$ and $\vph'(v_3)=b$.  Thus, $\LL$ mixes.

Finally, assume 
that $\vph(x)=\{a,b\}$ for all interior vertices $x$ of $P_1$, 
that $\vph(y)=\{a,c\}$ for all interior vertices $y$ of $P_2$, and 
that $\vph(z)=\{b,c\}$ for all interior vertices $z$ of $P_3$.
Similar to above,
$\LL_{v_1,c}\cup\LL_{w_1,c}\cup\LL_{v_2,b}\cup\LL_{w_2,b}\cup\LL_{v_3,a}\cup\LL_{w_3,a}$
mixes.
So assume that $\vph$ is not in this set of
colorings.  By symmetry (possibly swapping $P_2$ with $P_3$
and interchanging colors $a$ and $b$), we assume that $\vph(v_1)=a$,
$\vph(v_2)=c$, and $\vph(v_3)=b$.  In order to complete this coloring, the
lengths of $P_1$, $P_2$, and $P_3$ must have the same parity.

Suppose that $P_1$, $P_2$, and $P_3$ all have even length.  So
$\vph(w_1)=\vph(v_1)=a$, $\vph(w_2)=\vph(v_2)=c$, and $\vph(w_3)=\vph(v_3)=b$.
Now a $c,b$-swap at $v_2$ recolors all of $P_3$, as well as $v_2$
and $w_2$.  
This gives a new $L$-coloring $\vph'$ with
$\vph'(v_2)=\vph'(w_2)=b$, which mixes with all $L$-colorings above.
Assume instead that $P_1$, $P_2$, and $P_3$ all have odd length.
So $\vph(w_1)=b$, $\vph(w_2)=a$, and $\vph(w_3)=c$.  Now a $c,b$-swap at $v_2$
recolors all of $P_3$, as well as $v_2$ and $w_1$.
This gives a new $L$-coloring $\vph''$ with
$\vph''(v_2)=b$ and $\vph''(w_1)=c$, which mixes with all $L$-colorings above.
Thus, $\LL$ mixes.
\end{proof}

\section{Proofs of Structural Lemmas}
\label{structural-sec}
In this section, we prove our two structural lemmas.  For convenience, we
restate them.
Recall, for a plane graph $G$, that the subgraph \Emph{$G_3$} is
induced by all edges incident to vertices of degree 3.
If two vertices of degree 3 are adjacent, then $G$ contains an edge $vw$ such
that $d(v)+d(w)=6$, so the next lemma clearly holds.  Otherwise, the subgraph
$G_3$ must be bipartite.

\begin{struct-lemi}
If $G$ is a simple plane graph with $\delta\ge 2$, then either 
\begin{itemize}
\item[(C1)] $G$ contains an edge $vw$ with $d(v)+d(w)\le \max\{11,\DeltaG+2\}$; or
\item[(C2)] $G_3$ contains a bipartite barbell (possibly short); or
\item[(C3)] $G_3$ contains a bipartite $\theta$-graph other than $K_{2,3}$.
\end{itemize}
\end{struct-lemi}
Our proof is similar to a proof by Cohen and Havet~\cite{CH} of Borodin's result that
every simple planar graph with $\DeltaG\ge 9$ is $(\DeltaG+1)$-edge-choosable.
In fact, their proof implicitly contains an analogous structural lemma.
The main difference is that now we require either (C2) or (C3), whereas
they were satisfied with an even cycle in $G_3$.  Since (C2) and (C3) each
contain an even cycle, our version of the structural lemma is stronger.
\begin{proof}
Suppose the lemma is false, and let $G$ be a counterexample minimizing $\|G\|$.
We use discharging with initial charge $d(v)-4$ for each vertex $v$ and initial
charge $\ell(f)-4$ for each face $f$.   For each component of $G_3$, we also
use a ``pot'' with initial charge 0.  By Euler's formula, the sum of these initial
charges is $\sum_{v\in V(G)}(d(v)-4)+\sum_{f\in
F(G)}(\ell(f)-4)=4\|G\|-4|G|-4|F(G)|=-4(|F(G)|-\|G\|+|G|)=-8$.  So to reach a
contradiction, it suffices to discharge so that each vertex, face, and pot
finishes with nonnegative charge.  We use the following three discharging rules,
applied in succession.

\begin{enumerate}
\item[(R1)] Each $\Delta$-vertex with a 3-neighbor sends charge $1/2$ to
the pot for its component of $G_3$; each 3-vertex takes 1 from the pot for its
component of $G_3$.

\item[(R2)] Each $5^+$-vertex distributes its remaining charge equally, after
possibly applying (R1), among its incident faces (with multiplicity, if
applicable).

\item[(R3)] Each 3-vertex that lies on a 4-cycle in $G_3$ takes 1/2 from each
incident $4^+$-face and sends it to the pot for its component of $G_3$.
\end{enumerate}

Now we show that each vertex, face, and pot ends with charge nonnegative.
Note that $\delta\ge 3$, since $G$ has no instance of (C1).  If $v$ is a
3-vertex, then it takes charge 1 from the pot of its component, by (R1), and ends with charge
$3-4+1=0$.  (Note that (R3) does not affect the final charge of $v$.)  If $v$
is a 4-vertex, then it starts and ends with charge 0.  If $v$ is a
$5^+$-vertex, then it ends with charge 0, by (R2).  Next we consider faces.  Let
$f$ be a $4^+$-face. If $f$ is not incident to any 3-vertex, then it starts with charge
$\ell(f)-4\ge 0$ and ends with at least this much.  Suppose instead that $f$ has
an incident 3-vertex $v$.  Since $G$ has no instance of (C1), every neighbor of
$v$ along $f$ is a $\Delta$-vertex.  Further, $\Delta\ge 9$, since $\Delta+3>11$.  
Also, the number of $\Delta$-vertices incident to $f$ is at least the number of
incident 3-vertices (with multiplicity), and each $9^+$-vertex gives at least
$1/2$ to each incident face (see Claim~1 below for details).  So $f$ receives
at least as much charge from $\Delta$-vertices
as it gives to 3-vertices.  Hence $f$ ends with charge at least $\ell(f)-4\ge
0$.  So all that remain to consider are 3-faces and pots.
We bound the final charge of 3-faces via the following claim.

\begin{clm}
If a 3-face $f$ is incident to an $s$-vertex $v$ with $s\ge 5$, then $v$ sends $f$ at least
$1/5$ if $s\ge 5$, at least $1/3$ if $s\ge 6$, at least $3/7$ if $s\ge 7$, and
at least $1/2$ if $s\ge 8$.  
\end{clm}
\begin{clmproof}
If $s\le \Delta-1$ or $v$ has no 3-neighbor, then this is clear since each face
incident to $v$ receives at least $(s-4)/s$.  Now suppose that $s=\Delta$ and
$v$ has a 3-neighbor.  Since $G$ has no instance of (C1), this implies that $s\ge 9$. 
Now each face incident to $s$ receives $(s-4-1/2)/s=1-9/(2s)$; this is at least
$1/2$, since $s\ge 9$.  This proves the claim.
\end{clmproof}

Now let $f$ denote a 3-face and let $t$ denote the smallest degree of a vertex
incident to $f$.  Since $f$ starts with charge $3-4=-1$, to end nonnegative it must
receive at least 1.  Since $G$ has no instance of (C1), each edge $vw$ has
$d(v)+d(w)\ge 12$.
If $t\le 4$, then $f$ is also incident to two $8^+$-vertices,
so $f$ receives at least $0+2(1/2)=1$.  If
$t=5$, then $f$ is also incident to two $7^+$-vertices, so $f$ receives at least 
$1/5+2(3/7)>1$.  And if $t=6$, then $f$ receives at least $3(1/3)=1$.
Thus, each 3-face finishes with nonnegative charge.

Finally, consider a component $H$ of $G_3$.  Denote by $n_3$ and $n_{\Delta}$ the
numbers of 3-vertices and $\Delta$-vertices in $H$.  Note that each 3-vertex
in $H$ has all its incident edges in $H$, so $\|H\|=3n_3$.  First suppose
that $H$ contains at most one cycle.  Now $3n_3=\|H\|\le
|H|=n_3+n_{\Delta}$, so $2n_3\le n_{\Delta}$.  Thus, the pot for $H$ ends
nonnegative.  Next suppose that $H$ contains at least two distinct cycles. 
If $H$ contains two cycles that intersect in at most one vertex, then $H$
contains a barbell, formed from the these two cycles and a shortest path
joining them.  We are done, since (C2) holds.  

Suppose instead that every pair of cycles in $H$ intersect in at least two
vertices.  Now the union of any two cycles in $H$ contains a $\theta$-graph. 
If $G_3$ contains any $\theta$-graph other than $K_{2,3}$, then we are done,
since (C3) holds.  So assume it does not.  
Let $H'$ denote the subgraph of $H$ induced by edges on cycles in
$G_3$.  Note that $H'$ is $K_{2,t}$ for some integer $t$.  
To see this, consider an ear decomposition of $H'$, starting with an arbitrary
cycle.  If the cycle is not a 4-cycle or the first ear does not have length 2,
then the cycle plus the first ear is a $\theta$-graph other than $K_{2,3}$, as
desired.  So every cycle in $H'$ must be a 4-cycle (since it can start an ear
decomposition), and every ear must have length 2 (and so must connect two
vertices in the same part).  Further, all ears must connect the same two
vertices or we get a $\theta$-graph other than $K_{2,3}$.  This proves that
$H'=K_{2,t}$.

To prove that the pot for $H$ ends with nonnegative charge, we show that we can
direct the edges of $H$ so that each of its $\Delta$-vertices has outdegree
1, each of its 3-vertices has indegree at least 1 and each of its 3-vertices
has indegree 2 if it is either not in $H'$ or not on a $4^+$-face.
Given such an orientation, we are done, since we can view each 3-vertex as
receiving charge $1/2$ from each of its inneighbors, and from its incident
$4^+$-face, if it has one.

For each edge not in $H'$, we simply direct it toward $H'$.  This ensures
that each 3-vertex has indegree at least 1 and that each 3-vertex not in $H'$
has indegree 2.  First suppose that $H'$ contains at least three
$\Delta$-vertices.  In fact, it contains exactly three, and $H'=K_{2,3}$.
In this case, at least one 3-vertex $x$ in $H'$ will lie on a $4^+$-face.
We direct an edge from one $\Delta$-vertex to $x$, and direct edges from the
other two $\Delta$-vertices to the remaining 3-vertex in $H'$.

So assume instead that $H'$ contains exactly two $\Delta$-vertices;
denote these by $x$ and $y$.
Consider a vertex $z$ in $H'$ that is adjacent to $x$ and $y$.
Since $z$ has only one incident edge besides $xz$ and $yz$, the path $xzy$ must
be part of the boundary of exactly one face incident with $z$.  If that face
does not contain edge $xy$, then it is a $4^+$-face.  Since $G$ is simple, at
most two verices $z$ can lie on a 3-face, $xyz$.  We direct an edge from $x$ to
one of these two vertices and direct an edge from $y$ to the other.  This
completes the desired orientation of $H'$, which shows that the pot for $H$
ends with nonnegative charge.  It also finishes the proof of the lemma.
\end{proof}

Recall, for a plane graph $G$, that the subgraph \Emph{$G_2$} is
induced by all edges incident to vertices of degree 2 that lie on at least one
3-face.  If two such vertices are adjacent, then $G$ contains an edge $vw$
with $d(v)+d(w)=4$, so the lemma below is clearly true.
Otherwise, $G_2$ is bipartite (with all vertices of degree 2 in one part). 
As a result, we call every cycle in $G_2$ a \emph{2-alternating cycle}.

\begin{struct-lemi}
If $G$ is a simple plane graph with $\delta\ge 2$, then either 
\begin{itemize}
\item[(C1)] $G$ contains an edge $vw$ with $d(v)+d(w)\le 16$; or
\item[(C2)] $G_2$ contains a bipartite barbell (possibly short); or
\item[(C3)] $G_2$ contains either $K_{2,4}$ or a bipartite $\theta$-graph other than $K_{2,3}$.
\end{itemize}
\end{struct-lemi}

\begin{proof}
Assume the lemma is false, and let $G$ be a counterexample minimizing $\|G\|$.
Note that $d(v)+d(w)\ge 17$ for every edge $vw\in E(G)$ and that $G$ is connected. 
To reach a contradiction, we use discharging with initial charge $d(v)-4$ for each
vertex $v$ and $\ell(f)-4$ for each face $f$.  
For each component of $G_2$, we also have a ``pot'' with initial charge 0.  Since
these initial charges sum to $-8$, it suffices to discharge so that every
vertex, face, and pot ends with nonnegative charge.
We use the following five discharging rules, applied successively.
\begin{enumerate}
\item[(R1)] Each $15^+$-vertex $v$ with a 2-neighbor gives 1 to the pot for the
component of $G_2$ containing $v$, and each 2-vertex takes 1 from the pot for
its component of $G_2$.
\item[(R2)] Each $5^+$-vertex $v$ splits its remaining charge
equally among all incident faces (with multiplicity, if $v$ lies on the face
boundary multiple times).
\item[(R3)] Each 3-vertex takes 1/3 from each incident face (with multiplicity, if
applicable).
\item[(R4)] Each 2-vertex takes 1/3 from each incident 3-face, 2/3 from each incident
4-face with boundary that induces a 2-alternating cycle, and 1 from each other
incident $4^+$-face (with multiplicity, if applicable).
\item[(R5)] After applying all rules above, any positive charge at a 2-vertex is returned to
the pot for its component of $G_2$.
\end{enumerate}

Now we show that each vertex, face, and pot finishes with nonnegative charge.
Note that each $k$-vertex, with $k\ge 5$ gives each incident face at least $(k-4)/k$ or
$(k-5)/k$, the former unless $k\ge 15$.  Thus, each face $f$ receives at least
$1/3$ from each incident $6^+$-vertex and at least $2/3$ from each incident
$12^+$-vertex.

Each 2-vertex $v$ finishes with charge at least $2-4+1+(1/3)+(2/3)=0$, since $G$ is
simple, which implies that $v$ lies on at least one $4^+$-face.
Each 3-vertex $v$ finishes with charge $3-4+3(1/3)=0$.
Each 4-vertex finishes with charge $4-4=0$.
Each $5^+$-vertex finishes with charge 0, by (R2).

Let $f$ be a $4^+$-face.  Let $s$ denote the smallest degree of any vertex
incident to $f$.  If $s\ge 4$, then $f$ starts nonnegative, receives charge
from incident vertices, and does not give away charge, so ends positive.  If
$s=3$, then each incident 3-vertex is preceded along $f$ by a $14^+$-neighbor
(in some consistent direction); each such vertex pair gives $f$ a net charge of
at least $2/3-1/3$; so $f$ ends with positive charge.  Now assume that $s=2$. 
If $f$ is a 4-face and its boundary is a 2-alternating cycle, then $f$ finishes
with at least $4-4+2(2/3)-2(2/3)=0$.  Suppose instead that $f$ is a 4-face, but
its boundary does not induce a 2-alternating cycle.  The vertices incident to
$f$ include a single 2-vertex and at least two $15^+$-vertices.  Thus, $f$
finishes with at least $4-4+2(2/3)-1-1/3=0$.  Finally, suppose $f$ is a
$5^+$-face.  Now each incident $3^-$-vertex is preceded along $f$ (in some consistent
direction) by a $14^+$-vertex, and the net charge given to these two vertices
is at most 1/3.  Thus, $f$ finishes with at least $\ell(f)-4-\lfloor
\ell(f)/2\rfloor/3 \ge 5\ell(f)/6-4>0$.

Let $f$ be a 3-face, and let $s$ denote the smallest degree of any vertex
incident to $f$.  If $s\ge 6$, then each incident vertex gives $f$ at least
$1/3$, so $f$ finishes with at least $3-4+3(1/3)=0$.  So suppose that $s\le 5$.
 Now $f$ has a single incident $5^-$-vertex, which receives at most $1/3$ from
$f$.  Each other incident vertex gives $f$ at least $2/3$, so $f$ ends with at
least $3-4+2(2/3)-1/3=0$.

Finally, we consider the pot for a component $H$ of $G_2$.  If $H$ has at
least as many $\Delta$-vertices as 2-vertices, then the pot is nonnegative, by
(R1).  So assume this is not the case.  Thus, $H$ contains at least two cycles.
If two cycles in $H$ intersect in at most one vertex, then (C2) in the
statement of the lemma
holds (take a minimal connected subgraph containing the cycles).  So assume
that every pair of cycles in $H$ intersect in at least two vertices. 
Now let $H'$ be the union of all cycles in $H$. 
Since $H'$ is 2-connected, we consider an ear decomposition.
If $H'$ contains a cycle of any length other than 4, then we consider a
$\theta$-graph consisting of that cycle and the next ear in an ear
decomposition.  Now $G$ contains a $\theta$-graph other than $K_{2,3}$, so (C3)
holds.  So assume that $H'=K_{2,t}$ for some $t\ge 3$ (the analysis is similar
to the previous proof, but easier).  If $t\ge 4$, then (C3)
holds.  So assume that $H'=K_{2,3}$.
This implies that $H$ is formed from $K_{2,3}$ by adding one or more pendant
(even-length) paths at its $\Delta$-vertices.  Denote the vertex set of $H'$ by
$\{x_1,x_2, y_1,y_2,y_3\}$ where each $x_i$ is a $\Delta$-vertex in $G$ and
$d_{H'}(x_i)=3$ and $d_H(y_j)=2$ for all $i,j$.  

We consider the number of incidences $(y_i,f)$, where $i\in[3]$ and ($f$ ranges
over all faces).  Let $a_1$, $a_2$, $a_3$ denote the numbers of these
incidences, respectively, where $f$ is a 3-face, where $f$ is a 4-face with
boundary inducing a 2-alternating cycle, and where $f$ is another $4^+$-face.
If $(y_i,f)$ is counted by $a_1$, $a_2$, or $a_3$, then $f$ sends $y_i$ a charge
of $1/3$, $2/3$, or $3/3$ (respectively), by (R4).  We claim that $a_1\le a_3$.  This
implies that the total charge received by the $y_i$'s from their incident faces
is at least $6(2/3)=4$.  So the net charge that the $y_i$'s receive from their
pot is at most 2, as desired.

Now we prove that $a_1\le a_3$, as claimed.
Note that each face $f$ incident to some $y_i$ contains as part of its boundary
the path $x_1y_ix_2$.  If this boundary contains $x_1x_2$, then $f$ is a 3-face;
otherwise, $f$ is a $4^+$-face.  Since $G$ is simple, at most two of these
incident faces are 3-faces; that is, $a_1\le 2$.  If $a_1=0$, then we are done.
So assume that $x_1x_2\in E(G)$.
Consider the three 4-faces of $H'$, with boundaries $x_1y_ix_2y_j$ for each
pair of distinct $i,j\in[3]$.  The edge $x_1x_2$ lies inside one of these,
forming the two 3-cycles $x_1y_ix_2$ and $x_1y_jx_2$.  If both of these 3-cycles
are 3-faces of $G$, then the remaining edges incident to $x_1$ and $x_2$ must
each lie inside another of these 4-cycles.  So two of the $y_i$'s each
lie on some $4^+$-face that does not induce a 2-alternating cycle.  This is the
case that $a_1=2$ and $a_3\ge 2$.
Suppose instead that $a_1=1$.  Now again, the remaining edges incident to $x_1$
and $x_2$ must appear inside one of these 4-cycles, which implies that $a_3\ge
1$.  So in every case we have $a_1\le a_3$, as desired.
This proves the claim and finishes the lemma.
\end{proof}

\section*{Appendix: Proof of Lemma~1}
A key ingredient in the proofs of both Main Theorem~1 and Main Theorem~2 is
Lemma~\ref{CM-lem}, which was proved in~\cite{CM}.  However, since that
manuscript is not yet published (or refereed), for completeness we include a
proof below.  This proof is essentially copied from~\cite{CM}.  
First, we restate Lemma~1.

\setcounter{lemi}{0}
\begin{NoHyper}
\begin{lemi}
Fix a graph $G$ and a function $f:V(G)\to \Z^+$.  Let $H$ be an induced subgraph
of $G$ such that $G-H$ is $f$-swappable.  Let $f'(x):=f(x)-(d_G(x)-d_H(x))$ for
all $x\in V(H)$.  If $f'(x)\ge d_H(x)$ for all $x\in V(H)$ and $H$ is
both $f'$-choosable and $f'$-swappable, then $G$ is $f$-swappable.
\end{lemi}
\end{NoHyper}

\begin{defn}
For a graph $G$ and a list assignment $L$ for $G$, an $L$-coloring $\vph$ of $G$
is \emph{$\alpha,\beta$-versatile at $w$} if $\vph(w)\in\{\alpha,\beta\}$ and
an $\alpha,\beta$-swap at $w$ is $L$-valid for $\vph$.
\end{defn}

\begin{lem}
\label{versatile-lem}
Fix a graph $G$, a connected subgraph $H$, and a list assignment $L$ for $G$.
Let $\vph'$ be an $L$-coloring for $G-H$ that is $\alpha,\beta$-versatile at a
vertex $w$.  If $|L(v)|\ge d_G(v)$ for all $v\in V(H)$, and $H$ is not a Gallai
tree, then there exists an $L$-coloring $\vph$ of $G$ that extends $\vph'$ such
that $\vph$ is $\alpha,\beta$-versatile at $w$.  Further, there exists such an
$L$-coloring $\vph$ with the property that each $\alpha,\beta$-component of
$\vph$ contains the vertex set of at most one $\alpha,\beta$-component of
$\vph'$.
\end{lem}

\begin{proof}
Since $H$ is not a Gallai tree, it contains an induced even cycle, $C$,
with at most one chord.  We first show how to extend $\vph'$ to $G-C$, and then
how to extend it to all of $G$.

A key step is to show that if $|V(H)|\ge 2$ and $x\in V(H)$, then there exists
an $L$-coloring $\vph$ of $G-(H-x)$ that extends $\vph'$ to $x$ and that is
$\alpha,\beta$-versatile at $w$.  Suppose that $|V(H)|\ge 2$ and fix $x\in
V(H)$.  When choosing a color for $x$ (for brevity, we denote it by
$\vph'(x)$), to ensure that the resulting extension of $\vph'$ is an
$L$-coloring that is $\alpha,\beta$-versatile at $w$, we need to check the
following properties: (a) $\vph'(x)\ne \vph'(y)$ for all $y\in N(x)\setminus
V(H)$, (b) if $\vph'(x)\in\{\alpha,\beta\}$ and $x$ has a neighbor $y\notin H$
with $\vph'(y)\in\{\alpha,\beta\}$, then $\{\alpha,\beta\}\subseteq L(x)$, and
(c) if $\vph'(x)\in\{\alpha,\beta\}$, then $x$ has at most one neighbor
$y\notin H$ with $\vph'(y)\in\{\alpha,\beta\}$.  Now (a)
ensures the extension is a proper $L$-coloring; (b) ensures that an
$\alpha,\beta$-swap at $w$ will not create a problem at $x$; and (c)
ensures that each $\alpha,\beta$-component of $\vph$ contains at most one
$\alpha,\beta$-component of $\vph'$ and that an $\alpha,\beta$-swap at $x$
will not create a problem at $y$.

We form a list assignment $L'(x)$ from $L(x)$ by first removing each color that
is used by $\vph'$ on a neighbor of $x$.  Further, if $\alpha$ is used on a
neighbor of $x$ and $\alpha\notin L(x)$, then we remove $\beta$ from $L(x)$. 
Similarly, if $\alpha$ is used on two neighbors of $x$, then we remove
$\beta$ from $L(x)$, regardless of whether or not $\alpha\in L(x)$.  We also 
remove $\alpha$ from $L(x)$ if either of these situations occurs, but with
$\beta$ and $\alpha$ interchanged.  Since $|L(x)|\ge d_G(x)$, we must have
$|L'(x)|\ge d_H(x)\ge 1$, because $H$ is connected and $|V(H)|\ge 2$.  To extend
$\vph'$ to $x$, we simply choose any color in $L'(x)$.  This completes the key
step, started in the previous paragraph.  To extend $\vph'$ to $G-C$, we
repeatedly apply the key step, coloring vertices in order of non-increasing
distance from $C$.  Now we show how to extend $\vph'$ to $C$.  

First suppose that $C$ has no chord.  For each $v\in V(C)$, form $L'(v)$ as
in the previous paragraph.  Again, $|L'(v)|\ge 2$ for all $v\in V(C)$. 
First suppose there exists $\gamma$ and $x,y\in V(C)$ such that
$\gamma\notin\{\alpha,\beta\}$ and $xy\in E(C)$ and $\gamma\in L'(x)\setminus
L'(y)$.  Now we color $x$ with $\gamma$ and proceed around $C$ finishing with
$y$.  This process succeeds because each time we color another vertex $z_1$ we reduce
the number of allowable colors on its uncolored neighbor $z_2$ by at most one (even if
we completely repeat the process of removing colors as in the previous
paragraph, treating $z_1$ as though it is outside $H$).  We can finish at $y$
because the color on $x$ does not restrict our choice of color for $y$.
Suppose instead that there exists $\gamma\notin\{\alpha,\beta\}$ such that
$\gamma\in L'(v)$ for all $v\in C$.  Now we use $\gamma$ on one maximum
independent set in $C$ and color each remaining vertex $v$ from
$L'(v)\setminus\{\gamma\}$.  Finally, suppose that $L'(v)=\{\alpha,\beta\}$ for
each $v\in V(C)$.  Now we alternate $\alpha$ and $\beta$ around $C$.
In each case, it is easy to check that the resulting coloring $\vph$ is
$\alpha,\beta$-versatile at $w$.

Now suppose that $C$ has a chord.  Let $x$ denote one endpoint of the chord and
let $y$ and $z$ denote the neighbors of $x$ on $C$, besides the other endpoint
of the chord.  Form $L'(v)$ for each $v\in V(C)$, as above; again $|L'(v)|\ge
d_H(v)$ for all $v\in V(C)$.  Since $|L'(x)|\ge d_H(x)=3$, there exists
$\gamma\in L'(x)\setminus\{\alpha,\beta\}$.  If $\gamma\in L'(y)\cap L'(z)$,
then use $\gamma$ on $y$ and $z$ and color greedily toward $x$ in the remaining
uncolored subgraph.  So suppose instead that $\gamma\notin L'(y)\cap L'(z)$; by
symmetry, assume that $\gamma\notin L'(z)$.  Now use $\gamma$ on $x$, then color
the remaining uncolored subgraph greedily in order of non-increasing distance from
$z$.  Again, we can finish at $z$ because using $\gamma$ on $x$ does not
restrict the choice of color for $z$.
\end{proof}

Now we prove Lemma~\ref{CM-lem}.  As mentioned above, the proof mirrors that of
Lemma~\ref{degen-lem}, but now Lemma~\ref{versatile-lem} ensures that some
extension to $H$ is versatile for the next Kempe swap in $G-H$.

\begin{proof}[Proof of Lemma~\ref{CM-lem}.]
Fix $G$, $H$, $f$, and $f'$ that satisfy the hypotheses of
Lemma~\ref{CM-lem}.  Assume that $H$ is $f'$-choosable and
$f'$-swappable.  If $f'(w)>d_H(w)$ for some $w\in V(H)$, then
the proof is nearly the same as that of Corollary~\ref{degen-cor1}.
We delete the vertices of $H$ in order of non-decreasing distance from $w$.
Finally, we can handle all of $G-H$ at once because, by hypothesis, $G-H$ is
$f$-swappable.  So instead we assume that $f'(x)=d_H(x)$ for all $x\in V(H)$.
By assumption, $H$ is $f'$-choosable; this implies that $H$ is not a Gallai tree
(since all Gallai trees fail to be degree-choosable, as is easily shown by
induction on the number of blocks).  So now we will be able to apply
Lemma~\ref{versatile-lem} to $H$.

Let $L$ be an $f$-assignment for $G$.  Let $G':=G-H$.  By assumption, $G'$ is
$L$-swappable.  Let $\vph_0$ and $\vph$ be two $L$-colorings of $G$, and let
$\vph_0'$ and $\vph'$ denote their restrictions to $G'$.  Since $G'$ is
$L$-swappable, there exists a sequence $\vph_0', \vph_1', \ldots, \vph_k'$ of
$L$-colorings of $G'$, with $\vph_k'=\vph'$, such that every two successive
$L$-colorings differ by a single $L$-valid Kempe swap.  
By induction on $k$, we extend each
$\vph_i'$ to an $L$-coloring $\vph_i$ of $G$ such that every two successive
$L$-colorings in
the sequence $\vph_0,\vph_1,\ldots,\vph_k$ are $L$-equivalent. 

The case $k=0$ is easy, as follows. Form $L'$ from $L$ by removing, for each
$w\in H$, from $L(w)$ each color used by $\vph_0'$ on a neighbor of $w$ in
$G-H$.  Note that $|L'(w)|\ge f'(w)$ for all $w\in H$.  
So $H$ has an $L'$-coloring, since $H$ is $f'$-choosable.
So assume that $k\ge 1$, and fix $i$ such that $0\le i\le k-1$.
Suppose that $\vph_{i+1}'$ differs from $\vph_i'$ by an $\alpha,\beta$-swap at a
vertex $v_i$.  By Lemma~\ref{versatile-lem}, there exists an $L$-coloring
$\widetilde{\vph_i}$ of $G$, such that the restriction of $\widetilde{\vph_i}$
to $G'$ is $\vph'_i$ and an $\alpha,\beta$-swap at $v_i$ is $L$-valid in 
$\widetilde{\vph_i}$.  Furthermore, the restriction of this new coloring (after
performing the $\alpha,\beta$-swap at $v_i$) is $\vph_{i+1}'$.
It now suffices to show that $\vph_i$ and $\widetilde{\vph_i}$ are
$L$-equivalent. We do this by a sequence of Kempe swaps that recolors $H$ but
never changes the colors on $V(G-H)$.

For each $v\in V(G-H)$, remove $\vph'_{i}(v)$ from $L(w)$ for each $w\in
N(v)\cap V(H)$; denote the resulting list assignment on $H$ by $L_H$.  Note
that $|L_H(x)|\ge f'(x)\ge d_H(x)$ for all $x\in V(H)$.  If $|L_H(x)|>f'(x)$
for some $x$, then $H$ is $L_H$-swappable by Corollary~\ref{degen-cor1}, since
$|L_H(y)|\ge f'(y)\ge d_H(y)$ for all $y\in H$.  So assume $|L_H(y)|=f'(y)$ for
all $y\in H$.  Since $H$ is
$f'$-swappable, the restrictions of $\vph_i$ and $\widetilde{\vph_i}$ to $H$
(which are both $L_H$-colorings) are $L_H$-equivalent.  Consider a sequence of
Kempe swaps that witnesses this.  Note that performing the same Kempe swaps in
$G$ transforms $\vph_i$ to $\widetilde{\vph_i}$ (this is because for each edge
$vw$ with $v\in V(H)$ and $w\notin V(H)$, we have $\vph'_i(w)\notin L_H(v)$).
 Now performing an $\alpha,\beta$-swap at $v_i$ in $\widetilde{\vph_i}$ yields
an $L$-coloring of $G$ that restricts to $\vph'_{i+1}$ on $H$; we denote
this $L$-coloring of $G$ by $\vph_{i+1}$.

The previous two paragraphs show that we can use $L$-valid Kempe swaps to
transform $\vph_0$ into an $L$-coloring $\widetilde{\vph}$ that agrees with
$\vph$ on $G-H$.  Finally, we transform $\widetilde{\vph}$ to $\vph$.
This is possible precisely because $H$ is $f'$-swappable.
\end{proof}

\section*{Acknowledgment} Thanks to Reem Mahmoud and two referees for helpful feedback on 
earlier versions of this paper.

\bibliographystyle{habbrv}
{\small{\bibliography{edge-swappable}}}

\end{document}